\documentclass{article}
\author{}
\usepackage{todonotes}

\usepackage{amsmath, amscd, graphicx}
\usepackage{amssymb}
\usepackage{amsthm}
\usepackage{amsrefs}
\usepackage{cancel}
\usepackage[utf8]{inputenc}
\usepackage[T1]{fontenc}
\usepackage{amsmath}
\usepackage{amsfonts}
\usepackage{amssymb}
\usepackage[version=4]{mhchem}
\usepackage{stmaryrd}
\usepackage{bbold}
\usepackage{cases}
\usepackage{hyperref}
\usepackage{lipsum}

\DeclareSymbolFont{largesymbol}{OMX}{yhex}{m}{n}
\newtheorem{theo}{Theorem}[section]
\newtheorem{lemm}[theo]{Lemma}
\newtheorem{defi}[theo]{Definition}

\newtheorem{rema}{Remark}[section]
\numberwithin{equation}{section}

\begin{document}
	\title{Ill-posedness of the Kelvin-Helmholtz problem for compressible Euler fluids}
	
	\author{ Binqiang Xie $^{\dag}$, Bin Zhao $^{\ddag,*}$ \\[10pt]
		\small {$^\dag$ School of Mathematics and Statistics,}\\
		\small { Guangdong University of Technology, Guangzhou,   510006, China}
		\\
		\small {$^\ddag$  Institute of Applied Physics and Computational Mathematics, Beijing, 100088,  China}
	}
	
	\footnotetext{*Corresponding author.\\ E-mail addresses: \it xbqmath@gdut.edu.cn(B.Q. Xie),  \it zhaobin2017math@163.com(B. Zhao).}

	\date{}
	\maketitle
	\begin{abstract}
		In this paper, when the magnitude of the  Mach number is strictly between some fixed small enough constant and  $\sqrt{2}$, we can prove the linear and  nonlinear  ill-posedness of the Kelvin-Helmholtz problem for  compressible ideal  fluids. To our best knowledge, this is the first reslult that proves the nonlinear ill-posedness to the Kelvin-Helmholtz problem for the compressible Euler fluids. 
	\end{abstract}
	\section{Introduction}
	\subsection{Eulerian formulation}
	\quad \quad This paper concerns the Kelvin-Helmholtz problem for  compressible Euler fluids in the whole plane $\mathbb{R}^{2}$. More precisely, we consider two distinct invicid compressible, immiscible fluids evolving in the domain $\mathbb{R}^{2}$ for time $t\geq 0$. The fluids are separated from each other by a moving free surface $\Gamma(t)$, this surface divides $\mathbb{R}^{2}$ into two time-dependent, disjoint, open subsets $\Omega^{\pm}(t)$ such that $\Omega= \Omega^{+}(t) \cup \Omega^{-}(t) \cup \Gamma(t) $ and $\Gamma(t)=\bar{\Omega}^{+}(t) \cap \bar{\Omega}^{-}(t)$. The fluid occupying  $\Omega^{+}(t)$ is called the upper fluid and the second fluid, which occupies  $\Omega^{-}(t)$ is called the lower fluid. The two fluids  are sufficient  smooth to satisfy the pair of  compressible Euler equations:
	\begin{equation}
		\begin{cases}\label{1.1}
			\partial_t \rho^{\pm}+  \mathrm{div} (\rho^{\pm} u^{\pm}) =0,   \\
			\partial_t (\rho^{\pm} u^{\pm})+  \mathrm{div} (\rho^{\pm} u^{\pm}\otimes  u^{\pm}) +\nabla p^{\pm}=0,\\
		\end{cases}
	\end{equation}
	where  $u^{\pm}=(u_{1}^{\pm},u_{2}^{\pm})$ is the velocity field of the two fluids, $\rho ^{\pm}$ is the density of the two fluids,  $p^{\pm}$ denotes the  pressure of the two fluids in $\Omega^{\pm}$ respectively. We assume that $p$ is a $C^{\infty}$ function of $\rho$, defined on $(0,\infty)$ and such that $p^{\prime}(\rho)>0$ for all $\rho$. The speed of sound $c(\rho)$ in the fluid is defined by the relation:
	\begin{equation}\label{1.2}
		\forall \rho>0, ~c(\rho): = \sqrt{p^{\prime}(\rho)}.
	\end{equation}

	For the existence of weak solutions of \eqref{1.1} by the Rankine-Hugoniot jump relations of the hyperbolic system of equations, a standard assumption is that the pressure and the normal component of the  velocity must be continuous across the free boundary $\Gamma(t)=\{x_{2}= f(t,x_{1})\}$. Here the function $f$ describing the discontinuity front is part of the unknown of the problem, i.e. this is a free boundary problem.  Therefore, such piecewise smooth solution should satisfy the following
	boundary conditions on $\Gamma(t)$:
	\begin{equation} \label{1.5}
		\partial_{t}f= u^{+}\cdot n= u^{-}\cdot n, ~~p^{+}= p^{-},~~\mathrm{on}~~\Gamma(t),
	\end{equation}
	where $n=(-\partial_{x_{1}}f, 1)$ is the  normal vector to $\Gamma(t)$.

	To complete the statement of the problem, we must specify initial conditions. We give the initial interface $\Gamma_{0}$, which yields the open sets $\Omega^{\pm}_{0}$ on which we specify the initial density and velocity  field, $\rho^{\pm}(0,x)=\rho^{\pm}_{0}(x): \Omega_{0}^{\pm} \rightarrow \mathbb{R}^{+}$ and  $u^{\pm}(0,x)=u^{\pm}_{0}(x): \Omega_{0}^{\pm} \rightarrow \mathbb{R}^{2}$, respectively.

	Because $p^{\prime}(\rho)>0$, the function $p=p(\rho)$ can be inverted, allowing us to write $\rho=\rho(p)$. For  convenience in our subsequent analysis, given a  positive constant $\bar{\rho}$ defined in \eqref{1.888}, we introduce the quantity $E(p)=\log(\rho(p)/\bar{\rho})$ and consider $E$ as a new unknown quantity.  In terms of $(E,u)$, the system \eqref{1.1} is equivalent to the following equations:
	\begin{equation}\label{1.4}
		\begin{cases}
			\partial_{t} E+\left(u \cdot \nabla\right) E+  \nabla\cdot u=0, \\
			\partial_{t} u+\left(u \cdot \nabla\right) u+ c^{2} \nabla E=0.
		\end{cases}
	\end{equation}
	where the speed of sound is considered as a function of $E$, i.e., $c=c(E)$.
	
	The jump conditions \eqref{1.5} may be rewritten as 
	\begin{equation} \label{1.555}
		u^{+}\cdot n= u^{-}\cdot n, ~~E^{+}= E^{-},~~\mathrm{on}~~\Gamma(t).
	\end{equation}

		\subsection{Rectilinear solution}
\quad  It is easy to see that the system \eqref{1.1}-\eqref{1.5} admits a rectilinear solution $U=(\bar{f}, \rho^{\pm},\bar{u}^{\pm},)$ defined as following with the interface given by $\{x_{2}=0\}$ for all $t\geq0$. Then  $\Omega^{+}=\Omega^{+}(t)=\mathbb{R}\times (0,\infty)$ and  $\Omega^{-}=\Omega^{-}(t)=\mathbb{R}\times (-\infty,0)$ for all $t\geq0$.  More precisely, the front is flat, i.e., $\bar{f}=0$. To make sure the constant   density $\bar{\rho}^{\pm}$ satisfy the jump condition \eqref{1.5}, we must impose that 
\begin{equation} \label{1.888}
	\bar{\rho}^{+}=	\bar{\rho}^{-}:=\bar{\rho},
\end{equation}
where $\bar{\rho}$ is a positive constant. We also see that the upper fluid moves in the horizontal direction with some constant velocity and the lower fluid moves by the same constant velocity in the opposite direction, i.e, the  steady-state  constant velocity field $\bar{u}^{\pm}$  is the following form:
	\begin{equation}\label{1.3}
	\bar{u}=\left\{
		\begin{aligned}
			&(\bar{u}^{+}_{1},0)&x_2\geq0,\\
			&(\bar{u}^{-}_{1},0)&x_2<0,
		\end{aligned}
		\right.
	\end{equation}\label{1.777}
	where the constants $\bar{u}^{+}_{1},\bar{u}^{-}_{1}$ satisfy
	\begin{equation}
		\bar{u}^{+}_{1}=-\bar{u}^{-}_{1}.
	\end{equation}

	\subsection{New reformulation}
	\quad   Our analysis in this paper relies on the reformulation of the problem \eqref{1.4}-\eqref{1.555} under consideration in new coordinates.  To begin with, we define the fixed domains $\Omega^{\pm}$ as
	\begin{equation}\label{1.6}
		\begin{aligned}
			&\Omega^{+}:=\left\{x\in \mathbb{R}^{2}: x_{2}> 0\right\}, \\
			&\Omega^{-}:=\left\{x\in \mathbb{R}^{2}: x_{2}< 0\right\}.
		\end{aligned}
	\end{equation}
	Define the fixed  boundary $\Gamma$ as
	$$
	\Gamma:=\left\{x\in \mathbb{R}^{2}: x_{2}= 0\right\}.
	$$
	To reduce our free boundary problem to the fixed domain $\Omega^{\pm}$, we consider a change of variables on the whole space which maps $\Omega^{\pm}$ back to the origin domains $\Omega^{\pm}(t)$ by $(t,x_{1}, x_{2})\mapsto (t,x_{1}, x_{2}+ \psi(t,x))$.   We   construct such  $\psi$ by multiplying  the front $f$ by a smooth cut-off function depending on $x_{2}$:
	\begin{equation}\label{1.7}
		\psi(t,x_{1}, x_{2})= \theta(\frac{x_{2}}{3(1+a)})f(t,x_{1}),~a=\|f_{0}\|_{L^{\infty}(\mathbb{R})},
	\end{equation}
	where $\theta\in C^{\infty}_{c}(\mathbb{R})$ is a smooth cut-off function with $0\leq \theta\leq 1$, $\theta(x_{2})=1$, for $|x_{2}|\leq 1$, $\theta(x_{2})=0$ for $|x_{2}|\geq 3$, and $ |\partial_{2} \theta(x_{2})|\leq 1$ for all $x_{2}\in \mathbb{R}$,  writing $\partial_{j} = \partial/\partial x_{j}$.  We also assume 
	\begin{equation}\label{1.777}
 \|f_{0}\|_{L^{\infty}(\mathbb{R})}\leq 1.
		\end{equation}
Moreover, we have
	\begin{equation}\label{1.8}
		\begin{aligned}
			&\psi(t,x_{1}, 0,t)=f(t,x_{1}), \\
			&\partial_{2} \psi(t,x_{1}, 0)=0,\\
			&|\partial_{2} \psi | \leq \frac{1}{3(1+a)}|f|.
		\end{aligned}
	\end{equation}

	The change of variables that reduces the free boundary problem \eqref{1.1} to the fixed domain $\Omega^{\pm}$ is given in the following lemma.
	\begin{lemm}
		Define the function $\Psi$ by
		\begin{equation}\label{1.9}
			\Psi(t, x):=\left(x_{1}, x_{2}+\psi(t, x)\right), \quad(t, x) \in[0, T] \times \Omega.
		\end{equation}
		Then  $\Psi: (t,x)\mapsto (t,x_{1}, x_{2}+ \psi(t,x)) $ are diffeomorphism of $\Omega^{\pm}$ for all $t \in[0, T]$.
	\end{lemm}
	\begin{proof}
		Since $\|f_{0}\|_{L^{\infty}(\mathbb{R})}\leq 1$, one can prove that there exists some $T>0$ such that $ \sup_{[0,T]} \|f\|_{L^{\infty}}<2$, the free interface is still a graph within the time interval $[0,T]$	and
		$$
		\begin{aligned}
			\partial_{2} \Psi_{2}(t, x) & =1+\partial_{2} \psi(t, x) \geq 1-\frac{1}{3}\times  2 = \frac{1}{3},
		\end{aligned}
		$$
		which ensure that $\Psi: (t,x_{1}, x_{2})\mapsto (t,x_{1}, x_{2}+ \psi(t,x)) $ are diffeomorphism of $\Omega$ for all $t \in[0, T]$.
	\end{proof}

	We introduce the following operator notation
	$$
	\begin{array}{ll}
		A=[D \Psi]^{-1}=\left(\begin{array}{ccc}
			1 & 0  \\
			-\partial_{1} \psi / J  & 1 / J
		\end{array}\right),
	\end{array}
	$$
	$$
	a=J A =\left(\begin{array}{ccc}
		J & 0  \\
		-\partial_{1} \psi &  1,
	\end{array}\right)
	$$ and $ J=\operatorname{det}[D \Psi]=1+\partial_{2} \psi$.
	Now we may reduce the free boundary problem \eqref{1.4}-\eqref{1.555} to a problem in the fixed domain $\Omega^{\pm}$ by the map $\Psi$ defined in Lemma 1.1. Let us set 
		\begin{equation}\label{1.8888}
		\begin{aligned}
	&v^{ \pm}(t, x):=u^{ \pm}(t, \Psi(t, x)), ~\varrho^{ \pm}(t, x):=\rho^{\pm}(t, \Psi(t, x)), \\
	&q^{ \pm}(t, x):=p^{ \pm}(t, \Psi(t, x)),~h^{ \pm}(t, x):=E^{\pm}(t, \Psi(t, x)).
		\end{aligned}
\end{equation}
	Throughout the rest paper, an equation on $\Omega$ means that the equations holds in both $\Omega^{+}$ and $\Omega^{-}$.  For convenience, we consolidate notation by writing $\varrho$, $v$, $q$, $h$ to refer to $\varrho^{\pm},v^{\pm}, q^{\pm},h^{\pm}$ except when necessary to distinguish the two. Then system \eqref{1.4} and boundary conditions \eqref{1.555} can be reformulated on the fixed reference domain $\Omega^{\pm}$ as
	\begin{equation}\label{1.10}
		\begin{cases}
			\partial_{t} h+\left(\breve{v} \cdot \nabla\right) h+ A^{T} \nabla \cdot v=0,&~~in~\Omega, \\
			\partial_{t} v+\left(\breve{v} \cdot \nabla\right) v+ c^{2} A^{T} \nabla h=0,&~~in~\Omega, \\
			\partial_{t} f=v \cdot n,  &~~on~\Gamma, \\
			[v\cdot n]=0,\quad[h]=0, &~~on~\Gamma,\\
			v_{\mid t=0}=v_{0}, \quad h_{\mid t=0}=h_{0},&~~in~\Omega,\\
			f_{\mid t=0}=f_{0},&~~on~\Gamma,
		\end{cases}
	\end{equation}
	where  we set
	\begin{equation*}
		\begin{aligned}
			\breve{v} :=A v-\left(0, \partial_{t} \psi / J\right)=\left(v_{1}, \left(v \cdot n-\partial_{t} \psi\right) / J\right),
		\end{aligned}
	\end{equation*}
	and the notation $[h]= h^{+}|_{\Gamma}-h^{-}|_{\Gamma}$ denotes the jump of a quantity $h$ across $\Gamma$.
	
	The initial data are required to satisfy
	\begin{equation}\label{1.1111}
			\begin{aligned}
	&h^{+}_{0}=	h^{-}_{0}, &\text{in} ~\Omega,\\
	&v^{+}_{0}\cdot n_{0}=	v^{-}_{0}\cdot n_{0}, &\text{on}~ \Gamma.
			\end{aligned}
	\end{equation}

	Notice that
	\begin{equation}\label{1.11}
		J=1, \quad \breve{v}_{2}=0 \quad \text { on } \Gamma.
	\end{equation}

	Since we are interested in Kelvin-Helmholtz instability, the instability behavior firstly happens on the boundary.  To see this,  we are going to derive an second order evolution equation for the front $f$ on the fixed boundary $\Gamma$. By using the  momentum equation of  \eqref{1.10}, we deduce that
	\begin{equation}\label{1.12}
		\begin{aligned}
			\partial^{2}_{t} f =&\partial_{t} v^{+}\cdot n + v^{+} \cdot \partial_{t} n|_{\Gamma} \\
			=& -(\left(\breve{v}^{+} \cdot \nabla\right) v^{+}+ c^{2} A^{T} \nabla h^{ +})\cdot n- v^{+} \cdot (\partial_{1}\partial_{t} f, 0)|_{\Gamma}\\
			=& -v^{+}_{1}\partial_{1} v^{+}\cdot n-c^{2} A^{T} \nabla h^{ +}\cdot n
			-v^{+}_{1} \partial_{1}\partial_{t} f  |_{\Gamma}\\
			=& v^{+}_{1}\partial_{1}  n \cdot v^{+}- v^{+}_{1}\partial_{1} \partial_{t} f +  c^{2} A^{T} \nabla h^{+}\cdot n
			-v^{+}_{1} \partial_{1}\partial_{t} f  |_{\Gamma}\\
			=& -2v^{+}_{1}\partial_{1} \partial_{t} f  -  c^{2} A^{T} \nabla h^{+}\cdot n - (v_{1}^{+})^{2} \partial^{2}_{11} f|_{\Gamma}.
		\end{aligned}
	\end{equation}
	Similarly, we derive an  evolution equation for the front $f$ from the negative part:
	\begin{equation}\label{1.13}
		\begin{aligned}
			\partial^{2}_{t} f
			=-2v^{-}_{1}\partial_{1} \partial_{t} f  -  c^{2} A^{T} \nabla h^{-}\cdot n - (v_{1}^{-})^{2} \partial^{2}_{11} f~~~on~ \Gamma.
		\end{aligned}
	\end{equation}
	
	Therefore summing up the $"+"$ equation \eqref{1.12} and $"-"$ equation \eqref{1.13} to get
	\begin{equation}\label{1.14}
		\begin{aligned}
			&\partial^{2}_{tt} f+ (v^{+}_{1}+v^{-}_{1})\partial_{1} \partial_{t} f +\frac{1}{2}((c^{+})^{2} A^{T} \nabla h^{+}\cdot n+(c^{-})^{2} A^{T} \nabla h^{-}\cdot n)\\
			&+ \frac{1}{2}((v_{1}^{+})^{2}+(v_{1}^{-})^{2}) \partial^{2}_{11} f=0 ~~~on~ \Gamma.
		\end{aligned}
	\end{equation}

	\subsection{The wave equation for $h$}
	\quad
	
	Applying the operator $\partial_{t} +\breve{v} \cdot \nabla$ to the first equation of \eqref{1.10} and $A^{T}\nabla \cdot$ to the second one gives
	\begin{equation}\label{1.15}
		\left\{
		\begin{aligned}
			&(\partial_{t} +\breve{v}\cdot \nabla)^{2} h+ (\partial_{t} +\breve{v}\cdot \nabla)A^{T}\nabla \cdot v=0,\\
			&A^{T}\nabla \cdot (\partial_{t} +\breve{v}\cdot \nabla)v+ A^{T}\nabla \cdot (c^{2} A^{T}\nabla h)= 0.
		\end{aligned}
		\right.
	\end{equation}
	
	Next, we take the difference of the two equations in \eqref{1.15} to deduce a wave-type equation:
	\begin{equation}\label{1.16}
		(\partial_{t} +\breve{v}\cdot \nabla)^{2} h-  A^{T}\nabla \cdot (c^{2} A^{T}\nabla h)=\mathcal{F}.
	\end{equation}
	where the term $\mathcal{F}=-[\partial_{t}+\breve{v}\cdot \nabla, A^{T}\nabla\cdot ]v$  is a lower order term in the second order differential equation for $h$.
	
	From the boundary conditions in \eqref{1.10}, we already know that 
	\begin{equation}\label{1.17}
		[h]=0~~on~\Gamma.
	\end{equation}

	To determine the value of $h$, we add another condition involving the normal derivatives of $h$ on the boundary $\Gamma$. More precisely, Taking  the difference of  two equations \eqref{1.12} and  \eqref{1.13}, we can  obtain the jump of  $c^{2}A^{T} \nabla h\cdot n$,
	\begin{equation}\label{1.18}
		[c^{2}A^{T} \nabla h\cdot n]=[-2v_{1}\partial_{1}\partial_{t} f- (v_{1})^{2} \partial^{2}_{11} f]~~on~\Gamma.
	\end{equation}

	Combining \eqref{1.16},  \eqref{1.17} with \eqref{1.18} gives a nonlinear system for $h$:
	\begin{equation}\label{1.19}
		\left\{
		\begin{aligned}
			&	(\partial_{t} +\breve{v}\cdot \nabla)^{2} h-  A^{T}\nabla \cdot (c^{2} A^{T}\nabla h)=\mathcal{F}&~~in~\Omega,\\
			&[h]=0&~~on~\Gamma,\\
			&[c^{2}A^{T} \nabla h\cdot n]=[-2v_{1}\partial_{1}\partial_{t} f- (v_{1})^{2} \partial^{2}_{11} f]&~~on~\Gamma.
		\end{aligned}
		\right.
	\end{equation}
	
	\subsection{History result}
	\par \quad \quad  In Chandrasekhar's book \cite{Chandrasekhar}, the stability problem of superposed fluids can be divided into two kinds, the first kind of instability arises when two fluids of different densities superposed one over the other (heavy fluid over light fluid), is called Rayleigh-Taylor instability. There are lot of works about mathematical analysis of the Rayleigh-Taylor instability problem (\cite{Castro}, \cite{Guo11}, \cite{Guo1},\cite{Guo2},\cite{Guo3}, \cite{Jiang}). Ebin in \cite{E2} proved the instability for the Rayleigh-Taylor problem of  the incompressible Euler equation, while Guo and Tice in \cite{Guo1} showed the instability of this problem for the compressible inviscid case. Moreover, the Rayleigh-Taylor instability for the viscous compressible fluids was proved in \cite{Guo2} and for the inhomogeneous Euler equation in \cite{Guo3}. The second type of instability arises when the different layers of stratified heterogeneous fluid are in relative horizontal motion. In this paper, we study the second kind.
	
	\par The stability problems of two fluids in a relative motion have  attracted a wide interest of researchers of various fields. This type of instability is well known as the Kelvin-Helmholtz instability which was first studied
	by Hermann von Helmholtz in  \cite{Helmholtz} and by William Thomson (Lord Kelvin) in \cite{Kelvin}. The Kelvin-Helmholtz (K-H) instability is important in understanding a variety of space and astrophysical phenomena involving sheared plasma flow such as the stability of the
	interface between the solar wind and the magnetosphere (\cite{Dungey},\cite{Parker}), interaction between adjacent streams of different velocities in the solar wind \cite{Sturrock}  and the dynamic structure of cometary tails \cite{Ershkovich}.

	For Kelvin-Helmholtz instability in the incompressible Euler flow,  Ebin in \cite{E2} proved linear and nonlinear ill-posedness of the  well-known Kelvin-Helmholtz problem. O. B\"uhler, J. Shatah, S. Walsh and ChongChun Zeng in \cite{Buhle} gave a complete proof of the instability criterion and gave a unified equation connecting the Kelvin–Helmholtz and quasi-laminar for  the incompressible Euler flow.
	Recently we prove  linear and nonlinear ill-posedness of the  Kelvin-Helmholtz problem for incompressible MHD fluids \cite{Xie} under the condition violating the Syrovatskij stability condition. On the other hand, for Kelvin-Helmholtz instability in the compressible Euler flow,  by the normal  mode analysis, it is showed in  \cite {Landau},\cite{Fejer}, \cite{Miles} that the linear KH instability can be inhibitied  when the Mach number  $M:=\frac{\bar{u}^{+}_{1}}{c}>\sqrt{2}$ and the interface is violently unstable when $M< \sqrt{2}$. The Kelvin-Helmholtz instability configuration is also  known in literature as the ‘vortex sheet’, as its vorticity distribution is described by a $\delta$-function supported by a discontinuity in the velocity field at the sheet location. In the pioneer works \cite{Coulombel}, \cite{Secchi}, Coulombel and Secchi proved the  nonlinear stability of vortex sheets and local-in-time existence of two-dimensional supersonic vortex sheets by using a micro-local analysis and Nash-Moser method. Later on, Morando, Trebeschi and Wang \cite{Trebeschi1}, \cite{Trebeschi2} generalized this result to cover the two-dimensional nonisentropic flows. Our aim in this paper is to prove ill-posedness of Kelvin-Helmholtz problem  for  the nonlinear  Euler fluids
	exhibit the same ill-posedness as their linearized counterparts in \cite{Fejer}, \cite{Miles} under the condition $\epsilon_{0}\leq M:=\frac{\bar{u}^{+}_{1}}{c}<\sqrt{2}$, where $\epsilon_{0}$ is a small but fixed number.
	
	\subsection{Definitions and Terminology}
	\quad	Before stating the main result, we define some notation that will be throughout the paper. Throughout
	the paper $C>0$ will denote a generic constant that can depend on the parameters
	of the problem, but does not depend on the data, etc. We refer to such
	constants as “universal.” They are allowed to change from one inequality to the
	next. We will employ the notation $a\lesssim b$ to mean that $a\leq C b$  for a universal
	constant $C>0$. Also the notation $a\lesssim b$ denotes $a\leq C b$.
	Meanwhile , we will use $\mathfrak{R}$ to denote the real part of a complex number or a complex function.
	
	Since we study two disjoint fluids, for a function $\psi$ defined $\Omega$ we write $\psi_{+}$ for the restriction to $\Omega_{+}$ and $\psi_{-}$ for the restriction to $\Omega_{-}$. For all $j\in \mathbb{R}$, We  define the piecewise Sobolev space by 
	\begin{equation}\label{1.20}
		H^{j}(\Omega): =\{  \psi| \psi^{+} \in H^{j}(\Omega^{+}), \psi^{-} \in H^{j}(\Omega^{-}) \},
	\end{equation}
	endowed with the norm $\|\psi\|_{H^{j}}^{2}= \|\psi^{+}\|_{H^{j}(\Omega^{+})}^{2}+\|\psi^{-}\|_{H^{j}(\Omega^{-})}^{2}$.
	The usual Sobolev norm $ \|\psi\|_{H^{j}(\Omega^{\pm})}^{2}$ is equipped with the following norm:
	\begin{equation}\label{1.21}
		\begin{aligned}
			\|\psi\|_{H^{j}(\Omega^{\pm})}^{2}: &=\sum_{s=0}^{j} \int_{\mathbb{R}\times I_{\pm}} (1+\eta^{2})^{j-s} |\partial_{2}^{s}\hat{\psi}_{\pm}(\eta,x_{2})|^{2} d \eta d x_{2}\\
			&= \sum_{s=0}^{j} \int_{\mathbb{R}} (1+\eta^{2})^{j-s} \|\partial_{2}^{s}\hat{\psi}_{\pm}(\eta,x_{2})\|^{2}_{L^{2}(I_{\pm})}  d \eta,
		\end{aligned}
	\end{equation}
	where $I_{+}=(-\infty,0)$ and $I_{-}=(0,\infty)$ and $\hat{\psi}$ is the Fourier transform of $ f$ via
	\begin{equation}\label{1.22}
		\hat{\psi}(\eta)=  \int_{\mathbb{R}} \psi e^{-i x_{1} \eta}  dx_{1}, 
	\end{equation}
	for a function $\psi$ defined $\Gamma$, we define usual Sobolev space by 
	\begin{equation}\label{1.23}
		\|\psi\|_{H^{j}(\Gamma)}^{2}: =  \int_{\mathbb{R}} (1+\eta^{2})^{j} |\hat{\psi}(\eta)|^{2} d \eta.
	\end{equation}
	To shorten notation, for $j \geq 0$ we define
	\begin{equation}\label{1.24}
		\|(f, h, v)(t)\|_{H^{j}}=\|f(t)\|_{H^{j}(\Gamma)}+\|h(t)\|_{H^{j}(\Omega)}+\|v(t)\|_{H^{j}(\Omega)}. 
	\end{equation}
	
	\subsection{Main result}
	\quad  This paper is devoted to proving the ill-posedness of  Kelvin-Helmholtz problem for Euler system under the destabilizing effect of velocity shear violating the supersonic stability condition:
	\begin{equation}\label{1.25}
		\epsilon_{0} \leq M:=\frac{\bar{v}^{+}_{1}}{c}<\sqrt{2},
	\end{equation}
	where $\epsilon_{0}$ is a small but fixed number.
	
	\begin{defi}
		We say that the  problem \eqref{1.10} is locally well-posedness for some $k \geq 3$ if there exist $\delta, t_{0}, C>0$  such that for any initial data $(f_{0}^{1}, h_{0}^{1}, v_{0}^{1})$, $(f_{0}^{2}, h_{0}^{2}, v_{0}^{2})$ satisfying
		\begin{equation}\label{5.10}
			\|(f_{0}^{1}-f_{0}^{2}, h_{0}^{1}-h_{0}^{2}, v_{0}^{1}-v_{0}^{2})\|_{H^{k}}<\delta,
		\end{equation}
		there exist unique solutions$(f^{1}, h^{1}, v^{1})$ and  $(f^{2}, h^{2}, v^{2})\in L^{\infty}([0,t_{0}]; H^{3})$ of \eqref{1.10} with initial data $(f^{j}, h^{j}, v^{j})|_{t=0}=(f^{j}_{0}, h^{j}_{0}, v^{j}_{0})$ and there holds
		\begin{equation}\label{5.11}
			\begin{aligned}
				&\sup_{0\leq t \leq t_{0}}	\|(f^{1}-f^{2}, h^{1}-h^{2}, v^{1}-v^{2})(t)\|_{H^{3}}\\
				&\leq C(\| (f_{0}^{1}-f_{0}^{2}, h_{0}^{1}-h_{0}^{2}, v_{0}^{1}-v_{0}^{2}) \|_{H^{k}}).
			\end{aligned}
		\end{equation}
	\end{defi}

	\begin{theo}\label{theorem:main}
		Let the initial domain to be $\Omega_{0}= \Omega^{+}_0 \cup \Omega^{-}_0 \cup \Gamma_0$. Suppose that the initial data satisfies the constraint  condition \eqref{1.777} and \eqref{1.1111}, further we assume  the  rectilinear solution  satisfies the instability condition \eqref{1.25}.  Then the Kelvin-Helmholtz problem of \eqref{1.10}  is not locally well-posed in the sense of Definition 1.2.
	\end{theo}
	
	\begin{rema}
		We construct  the growing normal mode solution for the front $f$ when $0<M:=\frac{\bar{v}^{+}_{1}}{C}<\sqrt{2}$.  While for the linear and nonlinear problem, we only can prove the ill-posedness of the solutions $h, v$ of the Kelvin-Helmholtz problem to the ideal compressible  flow when $\epsilon_0<M:=\frac{\bar{v}^{+}_{1}}{C}<\sqrt{2}$ due to some 
		technical reason, where $\epsilon_0$ is some fixed small enough positive constant. 
	\end{rema}
	\begin{rema}
		Since $\Psi: (t,x)\mapsto (t,x_{1}, x_{2}+ \psi(t,x)) $ are diffeomorphism transform, the ill-poseness of system \eqref{1.10} in the flatten coordinates implies the ill-poseness of the solution to the original system \eqref{1.1}.
	\end{rema}
	
	\begin{rema}
		 Our results also hold for three-dimensional space case (\cite{Fejer}), the  instability condition \eqref{1.25} becomes to 
		 	\begin{equation*}
		 	\epsilon_{0} \leq M:=\frac{\bar{v}^{+}_{1} cos \phi}{c}<\sqrt{2},
		 \end{equation*}
		 where $\epsilon_{0}$ is a small but fixed number and $\phi$ is an angle between the displacement with equilibrium position.
	\end{rema}

	\section{The Linearized Equations in new coordinates}
	\quad \quad In this section, we consider  a  linearized system in new coordinates. We are going to  construct a  growing normal mode solution for this linearized system. By taking Fourier trnasform of linearized system, we get a second order ordinary equation for $\hat{g}$.

	\subsection{Construction of a growing solution of the  linearized system.}
	\quad \quad  It is easily verified that the particular solution in Euler coordinates is also a  particular solution in new coordinates such that
	\begin{equation} \label{2.1}
		\bar{v}^{\pm}=\bar{u}^{\pm}=\left\{
		\begin{aligned}
			&(\bar{v}^{+}_{1},0)&x_2\ge 0,\\
			&(\bar{v}^{-}_{1},0)&x_2<0,
		\end{aligned}
		\right.
	\end{equation}
	and 
		\begin{equation}\label{2.111}
	\bar{\varrho}^{+}=\bar{\varrho}^{-}:=\bar{\rho}.
	\end{equation}
	Now we will consider a constant coefficient linearized equations which is derived by linearizing the equation \eqref{1.10}  around the constant velocity $\bar{v}^{\pm}=(\bar{v}^{\pm}_{1},0)$, constant pressure $\bar{h}^{+}=\bar{h}^{-}$, and flat front $\Gamma=\{x_{2}=0 \}$, i.e. $\bar{f}=0$, the  outer normal vector $\bar{n}=(0,1):=e_{2}$. The resulting linearized equations are
	\begin{equation}\label{2.2}
		\begin{cases}
			\partial_{t} h+\bar{v}_{1} \partial_{1} h+ {\rm div} v=0, & \text { in } \Omega, \\
			\partial_{t} v+\bar{v}_{1} \partial_{1} v+ c^{2} \nabla h=0, & \text { in } \Omega, \\
			\partial_{t} f=v_{2}-\bar{v}_{1} \partial_{1} f,   & \text { on } \Gamma.
		\end{cases}
	\end{equation}
	where $\bar{v}_{1}^{\pm}$ and $c^{2}=c^{2}(\bar{h})$ are constants.
	In order to linearize the jump conditions in \eqref{1.10}, we let $v=\bar{v}+ \tilde{v}$ and $n=e_{2}+ \tilde{n}$, we linearize the origin boundary condition $ [v\cdot n]=0$ as follows:
	\begin{equation*}
		[(\bar{v}+ \tilde{v})\cdot (e_{2}+ \tilde{n}) ]=[\tilde{v} \cdot e_{2}]+ [\bar{v}\cdot \tilde{n} ]+ [\tilde{v} \cdot \tilde{n}]=0,
	\end{equation*}
	where $\tilde{n}=(-\partial_{1}\tilde{f},0 )$. Obviously, the third term is nonlinear term, it follows that 
	\begin{equation*}
		[\tilde{v} \cdot e_{2}] =- [\bar{v}\cdot \tilde{n} ]= 2\bar{v}^{+}_{1}\partial_{1} \tilde{f}.
	\end{equation*}
	Thus, the jump conditions on the boundary linearize to 
	\begin{equation}\label{2.3}
		[h]=0,~[v\cdot e_{2}]= 2\bar{v}^{+}_{1}\partial_{1} f.
	\end{equation}

	We also get a linearized equation for the front $f$
	\begin{equation}\label{2.4}
		\begin{aligned}
			&\partial_{t}^{2}f+ (\bar{v}^{+}_{1})^{2} \partial^{2}_{11} f+ \frac{c^{2}}{2} \partial_{2} (h^{+}+h^{-})=0 ~~~on~ \Gamma,
		\end{aligned}
	\end{equation}
	and a linearized system for  the pressure $h$
	\begin{equation}\label{2.5}
		\left\{
		\begin{aligned}
			&	(\partial_{t} +\bar{v}_{1} \partial_{1})^{2} h- c^{2} \Delta h=0&~~on~ \Omega,\\
			&[h]=0&~~on~\Gamma,\\
			&[c^{2}\partial_{2} h]=-4\bar{v}^{+}_{1}\partial_{1}\partial_{t} f&~~on~ \Gamma.
		\end{aligned}
		\right.
	\end{equation}
	
	Since we want to construct a solution to the linear system \eqref{2.2}-\eqref{2.5} that has a growing $H^{k}$ norm for any $k$, i.e., we assume the solution is in the following normal mode form:
	\begin{equation}\label{2.6}
		h(t,x_{1},x_{2})= e^{\tau t} m(x_{1},x_{2}), v(t,x_{1},x_{2})= e^{\tau t} w(x_{1},x_{2}), f(t,x_{1})= e^{\tau t} g(x_{1}),
	\end{equation}
	here we assume that $\tau=\gamma+ i \delta \in \mathbb{C}\backslash \{0\}  $ is the same above and below the interface. A solution with $\mathfrak{R} (\tau)>0$ corresponds to a growing mode. Plugging the ansatz \eqref{2.6} into  \eqref{2.2}-\eqref{2.5}, we have 
	\begin{equation}\label{2.7}
		\begin{cases}
			\tau  m+\bar{v}_{1} \partial_{1} m+ {\rm div} w=0, & \text { in } \Omega, \\
			\tau w+\bar{v}_{1} \partial_{1} w+ c^{2} \nabla m=0, & \text { in } \Omega, \\
			\tau g=w_{2}-\bar{v}_{1} \partial_{1} g, \quad  [w\cdot e_{2}]= 2\bar{v}^{+}_{1}\partial_{1} g, \quad[m]=0, & \text { on } \Gamma, 
		\end{cases}
	\end{equation}
	and
	\begin{equation}\label{2.8}
		\begin{aligned}
			&\tau^{2}g+ (\bar{v}^{+}_{1})^{2} \partial^{2}_{11} g+ \frac{c^{2}}{2} \partial_{2} (m^{+}+m^{-})=0 ~~~on~ \Gamma,
		\end{aligned}
	\end{equation}
	and  
	\begin{equation}\label{2.9}
		\left\{
		\begin{aligned}
			&	(\tau +\bar{v}_{1} \partial_{1})^{2} m- c^{2} \Delta m=0&~~on~ \Omega,\\
			&[m]=0&~~on~ \Gamma,\\
			&[c^{2}\partial_{2} m]=-4\bar{v}^{+}_{1} \tau \partial_{1} g&~~on~\Gamma.
		\end{aligned}
		\right.
	\end{equation}
	
	\subsection{ The formula for $\partial_{2}\hat{m}^{+}+\partial_{2}\hat{m}^{-}$.}
	\quad  \quad  We  take the horizontal  Fourier transform to the equation \eqref{2.8} and \eqref{2.9} and  deduce the formula for $\partial_{2}\hat{m}^{+}+\partial_{2}\hat{m}^{-}$ on $\Gamma$, then substituting this formula into \eqref{2.8}, therefore we have an second-order equation for $g$ without coupling with other quantity.  To begin with,  we define the Fourier transform of $h$ and $f$ as follow:
	\begin{equation*}
		\hat{m}(\eta,x_{2})= \int_{\mathbb{R}} m(x_{1},x_{2}) e^{-ix_{1} \eta}  dx_{1} ,~ \hat{g}(\eta)= \int_{\mathbb{R}} g(x_{1}) e^{-i x_{1} \eta} dx_{1}.
	\end{equation*}

	Taking the Fourier transform to the equation \eqref{2.8}, \eqref{2.9} with respect with the horizontal variable to get
	\begin{equation}\label{2.10}
		\begin{aligned}
			\tau^{2}\hat{g}-(\bar{v}^{+}_{1})^{2} \eta^{2} \hat{g}+ \frac{c^{2}}{2} \partial_{2} (\hat{m}^{+}+\hat{m}^{-})=0~~on~ \Gamma,
		\end{aligned}
	\end{equation}
	and
	\begin{equation}\label{2.11}
		\left\{
		\begin{aligned}
			&(\tau +i\bar{v}_{1}\eta)^{2} \hat{m}+c^{2}\eta^{2} \hat{m}-  c^{2} \partial^{2}_{22} \hat{m}=0&~~on~ \Omega,\\
			&[\hat{m}]=0&~~on~ \Gamma,\\
			&[c^{2}\partial_{2} \hat{m}]= -4 i\bar{v}^{+}_{1} \eta \tau \hat{g}&~~on~ \Gamma.
		\end{aligned}
		\right.
	\end{equation}
	
	Solving the system  \eqref{2.11}, we obtain
	\begin{equation} \label{2.12}
		\hat{m}(\eta,x_{2})=\left\{
		\begin{aligned}
			& \frac{4 i\bar{v}^{+}_{1} \eta \tau \hat{g}}{c^{2}(\mu^{+}+\mu^{-})}e^{-\mu^{+}x_2}&x_2\ge 0,\\
			& \frac{4 i\bar{v}^{+}_{1} \eta \tau \hat{g}}{c^{2}(\mu^{+}+\mu^{-})}e^{\mu^{-}x_2}&x_2<0,
		\end{aligned}
		\right.
	\end{equation}
	where $\mu^{\pm}=\sqrt{\frac{(\tau\pm i\bar{v}^{+}_{1} \eta)^{2}}{c^{2}} + \eta^{2}}$ are the root of the equation
	\begin{equation}\label{2.13}
		c^{2} s^{2}-(\tau +i\bar{v}^{\pm}_{1}\eta)^{2}- c^{2}\eta^{2}=0,
	\end{equation}
	here we notice that $\mathfrak{R} \mu^{\pm}>0$ since $\mathfrak{R} \tau>0$.
	
	By direct computation, we can arrive at
	\begin{equation}\label{2.14}
		\partial_{2}\hat{m}^{+}+\partial_{2}\hat{m}^{-}=-\frac{4 i\bar{v}^{+}_{1} \eta \tau}{c^{2}}\hat{g} \frac{\mu^{+}-\mu^{-}}{\mu^{+}+\mu^{-}},~~on~ \Gamma.
	\end{equation}

	Plugging \eqref{2.14} into \eqref{2.10}, we get an second order equation for $\hat{g}$
	\begin{equation}\label{2.15}
		\begin{aligned}
			(\tau^{2}-(\bar{v}^{+}_{1})^{2} \eta^{2}-2i \bar{v}^{+}_{1} \eta \tau \frac{\mu^{+}-\mu^{-}}{\mu^{+}+\mu^{-}})\hat{g}=0,~~on~\Gamma.
		\end{aligned}
	\end{equation}
	Finally, the symbol of \eqref{2.15} is defined as follows
	\begin{equation}\label{3.1}
		\Sigma:= \tau^{2}-(\bar{v}^{+}_{1})^{2} \eta^{2}-2i \bar{v}^{+}_{1} \eta \tau \frac{\mu^{+}-\mu^{-}}{\mu^{+}+\mu^{-}}.
	\end{equation}

	\section{The  analysis of the  symbol  \eqref{3.1}}
	\quad  In this section, the analysis of the  symbol \eqref{3.1} is established in  the spirit of  Morando-Secchi-Trebeschi \cite{Morando}. To begin with, we define a set of "frequencies"
	\begin{equation}\label{3.2}
		\Xi= \{ (\tau,\eta)\in \mathbb{C} \times \mathbb{R}  : \mathfrak{R} \tau> 0, (\tau,\eta)\neq (0,0)\}.
	\end{equation}
	
	Since we already know that $\mathfrak{R} \mu^{\pm}>0$ in all points with $\mathfrak{R} \tau>0$. It follows that $\mathfrak{R} (\mu^{+}+ \mu^{-})>0$ and thus $\mu^{+}+ \mu^{-}>0$ in all such points. From \eqref{3.1}, the symbol $\Sigma$ is  defined in points $(\tau,\eta)\in \Xi$.
	
	We also need to know whether the difference $\mu^{+}-\mu^{-}$ vanishes.
	\begin{lemm}
		Let $(\tau,\eta)\in \Xi$. Then $\mu^{+}=\mu^{-}$ if and only if $(\tau,\eta)=(\tau,0)$.
	\end{lemm}
	\begin{proof}
		From \eqref{2.13}, it implies that  $(\mu^{+})^{2}=(\mu^{-})^{2}$ if and only if $\eta=0$ or $\tau=0$. Since  $(\tau,\eta)\in \Xi$, only  $\eta=0$ case need to study. When $\eta=0$, it follows that $\mu^{+}=\mu^{-}= \tau/c$.
	\end{proof}

	Now we will discuss the roots of  the symbol \eqref{3.1} in the instability case.
	\begin{lemm}
		Let $\Sigma(\tau,\eta)$ be the symbol defined in \eqref{3.1}, for $(\tau, \eta) \in \Xi$.
		If $\bar{v}^{+}_{1}<\sqrt{2}c$, then $\Sigma(\tau,\eta)=0$ if only if
		\begin{equation}\label{3.27}
			\tau=  X_{1} \eta,
		\end{equation}
		where $	X_{1}^{2}=\sqrt{c^{4}+ 4c^{2}(\bar{v}^{+}_{1})^{2}}-(\bar{v}^{+}_{1})^{2}-c^{2}>0$. The root $	\tau=  X_{1} \eta$ is simple, i.e. there exists a neighborhood $\mathcal{V}$ of $(X_{1}\eta, \eta)\in \Xi$ and a smooth $F$ defined on  $\mathcal{V}$ such that 
		\begin{equation}\label{3.28}
			\Sigma= (\tau-X_{1}\eta) F(\tau,\eta),~F(\tau,\eta)\neq 0 ~\text{for all} (\tau,\eta)\in \mathcal{V}, 
		\end{equation}
		where $F(\tau, \eta)$ is defined as $c^{2}\eta^{2} \frac{d\phi }{d X}(\alpha X_1+ (1-\alpha) X))$.
	\end{lemm}
	
	\begin{proof}
		In according with the definition of $\Sigma$ and Lemma 3.1,  we can easily  verify $\Sigma(\tau,0)=\tau^{2}\neq 0$ for $(\tau,0)\in \Xi$.  Meanwhile, it is easy to check that $\Sigma(\tau, \eta)= \Sigma(\tau, -\eta)$. Thus we can assume without loss of generality that $\tau\neq 0$, $\eta\neq 0 $ and $ \eta>0$ and from Lemma 3.1 we know that $\mu^{+}- \mu^{-}\neq 0$. Therefore we compute
		\begin{equation}\label{3.22}
			\begin{aligned}
				\frac{\mu^{+}-\mu^{-}}{\mu^{+}+\mu^{-}}= \frac{(\mu^{+}-\mu^{-})^{2}}{(\mu^{+})^{2}-(\mu^{-})^{2}}= \frac{c^{2}(\mu^{+}-\mu^{-})^{2}}{4 i\bar{v}^{+}_{1}\eta \tau},
			\end{aligned}
		\end{equation}
		and
		\begin{equation}\label{3.23}
			\begin{aligned}
				(\mu^{+}-\mu^{-})^{2}= 2((\frac{\tau}{c})^{2}- (\frac{\bar{v}^{+}_{1}\eta}{c})^{2} +\eta^{2}- \mu^{+}\mu^{-}),
			\end{aligned}
		\end{equation}
		therefore we deduce that
		\begin{equation}\label{3.24}
			\begin{aligned}
				\frac{\mu^{+}-\mu^{-}}{\mu^{+}+\mu^{-}}=\frac{{\tau}^{2}- (\bar{v}^{+}_{1}\eta)^{2} +c^{2}(\eta^{2}- \mu^{+}\mu^{-})}{2 i\bar{v}^{+}_{1}\eta \tau },
			\end{aligned}
		\end{equation}
		and substituting the formula \eqref{3.24} into \eqref{2.15} we can rewrite it as
		\begin{equation}\label{3.25}
			\begin{aligned}
				c^{2}( \mu^{+}\mu^{-}-\eta^{2})\hat{g}=0,~~\text{on}~\Gamma,
			\end{aligned}
		\end{equation}
		the symbol $\Sigma$ can be  reformulated as
		\begin{equation}\label{3.26}
			\begin{aligned}
				\Sigma= c^{2}( \mu^{+}\mu^{-}-\eta^{2}).
			\end{aligned}
		\end{equation}

		Let us set $\mu^{+}\mu^{-}-\eta^{2}=0$ and introduce the quantities:
		\begin{equation}\label{3.29}
			X=\frac{\tau}{\eta},~\tilde{\mu}^{\pm}= \frac{\mu^{\pm}}{ \eta}.
		\end{equation}
		Therefore we can deduce
		\begin{equation} \label{3.30}
			\tilde{\mu}^{+}\tilde{\mu}^{-}=1,
		\end{equation}
		and
		\begin{equation}\label{3.31}
			(\tilde{\mu}^{+})^{2}(\tilde{\mu}^{-})^{2}=1.
		\end{equation}
		
		By the formula of the roots $\mu^{\pm}$, it follows that
		\begin{equation}\label{3.32}
			(\tilde{\mu}^{+})^{2}=\frac{(X+i\bar{v}^{+}_{1})^{2}}{c^{2}} +1,
		\end{equation}
		and
		\begin{equation}\label{3.33}
			(\tilde{\mu}^{-})^{2}=\frac{(X-i\bar{v}^{+}_{1})^{2}}{c^{2}}+1,
		\end{equation}
		then substituting \eqref{3.32} and \eqref{3.33} into \eqref{3.31}, we get
		\begin{equation}\label{3.34}
			[(X+i\bar{v}^{+}_{1})^{2}+c^{2}][(X-i\bar{v}^{+}_{1})^{2}+c^{2}]=c^{4},
		\end{equation}
		which leads to an quadratic equation for $X^{2}$:
		\begin{equation}\label{3.35}
			\begin{aligned}
				X^{4} + 2 ((\bar{v}^{+}_{1})^{2}+c^{2}) X^{2}+ (\bar{v}^{+}_{1})^{4} -2c^{2} (\bar{v}^{+}_{1})^{2} =0.
			\end{aligned}
		\end{equation}
		Using the quadratic root formula, the two roots of the  equation \eqref{3.35} are
		\begin{equation}\label{3.36}
			X_{1}^{2}=- (\bar{v}^{+}_{1})^{2}-c^{2}+ \sqrt{c^{4}+ 4c^{2}(\bar{v}^{+}_{1})^{2}},
		\end{equation}
		and
		\begin{equation}\label{3.37}
			X_{2}^{2}= -(\bar{v}^{+}_{1})^{2}-c^{2}-  \sqrt{c^{4}+ 4c^{2}(\bar{v}^{+}_{1})^{2}},
		\end{equation}
		We claim that the points $(\tau,\eta)\in \Sigma$ with $\tau=\pm  X_{2} \eta$ are not the roots of $\mu^{+}\mu^{-}= \eta^{2}$. Without loss of generality, we can assume that $Y_{2}$ is positive. From \eqref{3.37}, we deduce
		\begin{equation}\label{3.38}
			X_{2}=iY_{2},~Y_{2}= \sqrt{(\bar{v}^{+}_{1})^{2}+c^{2}+\sqrt{c^{4}+ 4c^{2}(\bar{v}^{+}_{1})^{2}}}\geq \bar{v}^{+}_{1}+c,
		\end{equation}
		from this  we deduce $ Y_{2}\pm \bar{v}^{+}_{1}>c$., in accord with the equation \eqref{3.32} and \eqref{3.33}, we deduce that $\tilde{\mu}^{+}= i \sqrt{\frac{(Y_{2}+\bar{v}^{+}_{1})^{2}}{c^{2}} -1},\tilde{\mu}^{-}=  i \sqrt{\frac{(Y_{2}-\bar{v}^{+}_{1})^{2}}{c^{2}} -1}$
		from which we know that   $\tilde{\mu}^{+}\tilde{\mu}^{-}=1$  is not satisfied. Similarly, we can show that $(\tau,\eta)\in \Sigma$ with $\tau=- X_{2}\eta$ is not root of $\mu^{+}\mu^{-}= \eta^{2}$. On the other hand, from \eqref{3.38}, we know that $\tau=iY_{2}\eta$ is imaginary root, thus it implies that $\mathfrak{R} \tau=0$ and  $(\pm X_{2}\eta, \eta)\nsubseteq \Xi$.

		Now we  focus on  the root $X_{1}^{2}$. If $\bar{v}^{+}_{1}<\sqrt{2}c$, from \eqref{3.36}, we know that $X_{1}^{2}$ is positive, it follows that $\tau = \pm  X_{1} \eta$ are real. The point $(-  X_{1} \eta,\eta)\nsubseteq \Xi$, thus we  omit this point, we only study the root $\tau = +  X_{1} \eta$. Using a fact that square roots of the complex number $a+ib$ are 
		\begin{equation}\label{3.39}
			\pm\{\sqrt{\frac{r+a}{2}} + i sgn(b)\sqrt{\frac{r-a}{2}}\},~r=|a+ib|,
		\end{equation}
		in our case we compute
		\begin{equation}\label{3.40}
			\mu^{+}=\sqrt{\frac{r+a}{2}} + i\sqrt{\frac{r-a}{2}},\mu^{-}=\sqrt{\frac{r+a}{2}} - i\sqrt{\frac{r-a}{2}},
		\end{equation}
		where 
		\begin{equation}\label{3.41}
			a=\frac{X^{2}_{1}- (\bar{v}^{+}_{1})^{2}+ c^{2}}{c^{2}} \eta^{2},~b= \frac{2X_{1}\bar{v}^{+}}{c^{2}} \eta^{2},
		\end{equation}
		so that $	\mu^{+}	\mu^{-}=r>0$, therefore we deduce that in case of $\bar{v}^{+}_{1}<\sqrt{2}c$,    the  root of  the symbol $\Sigma$ is  the point $(+ X_{1}\eta, \eta)$. In summary we can get a root $(\tau,\eta)$ with $\mathfrak{R} \tau>0$, which is a unstable solution. 
		
		Now we prove that the root $( X_{1}\eta,\eta)$ are simple. We define $\phi(X) = \tilde{\mu}^{+} \tilde{\mu}^{-}-1$, therefore we have $\Sigma=c^{2}\eta^{2}\phi(X) $. By Taylor formula, we can write 
		\begin{equation}\label{3.42}
			\Sigma= c^{2}\eta^{2} (\phi(X_{1})+ (X-X_{1}) \frac{d\phi }{d X}(\alpha X_1+ (1-\alpha) X)), 0<\alpha<1,
		\end{equation}
		by direct computation, we have 
		\begin{equation}\label{3.43}
			\phi(X_{1})= 0,~ \frac{d\phi }{d X}= \frac{2X/c}{\tilde{\mu}^{+} \tilde{\mu}^{-}} \{ (\frac{X}{c})^{2}+ (\bar{v}^{+}_{1}/c)^{2}+1\}.
		\end{equation}
		Since $\frac{d\phi }{d X} (X_{1})\neq  0$, by the continuity of  $\frac{d\phi }{d X}$, it follows that $\frac{d\phi }{d X} (\alpha X_1+ (1-\alpha) X))\neq  0$. Therefore we complete the proof of this lemma.
	\end{proof}
	
	\section{Ill-posedness of  solutions for the linear problem}
	\subsection{Uniqueness for the linearized equations \eqref{2.2}}
	\quad To begin with, we prove a uniqueness result for the linearized equations \eqref{2.2}. 
	\begin{lemm}
		Let $f,h,v$ be a solution to the linearized equations \eqref{2.2} with  $(f,h,v)|_{t=0}=0$. Then $(f,h,v)\equiv0$.
	\end{lemm}
	\begin{proof}
		Taking  the standard inner product of the first and second equations in \eqref{2.2} with $h^{+}, v^{+}$ and integrating over $\Omega^{+}$, we obtain
		\begin{equation}\label{4.1}
			\frac{1}{2}	\partial_{t} \int_{\Omega^{+}}c^{2} |h^{+}|^{2}+\frac{1}{2} \int_{\Omega^{+}}\bar{v}_{1} \partial_{1}(c^{2} |h^{+}|^{2}) + \int_{\Omega^{+}} c^{2} h^{+} {\rm div} v^{+}=0  .
		\end{equation}
		and
		\begin{equation}\label{4.2}
			\frac{1}{2}	\partial_{t}  \int_{\Omega^{+}}|v^{+}|^{2}+\frac{1}{2} \int_{\Omega^{+}}\bar{v}^{+}_{1} \partial_{1}(|v^{+}|^{2}) + \int_{\Omega^{+}} c^{2} \nabla h^{+} v^{+}=0.
		\end{equation}
		The second terms on the left hand side of  \eqref{4.1} and \eqref{4.2} vanish, thus adding \eqref{4.1} and \eqref{4.2} and integrating by parts, we get
		\begin{equation}\label{4.111}
			\frac{1}{2}	\partial_{t} \int_{\Omega^{+}}(c^{2} |h^{+}|^{2}+|v^{+}|^{2} ) =c^{2}\int_{\Gamma} h^{+} v^{+} \cdot e_{2}  .
		\end{equation}
		A similar result holds on $\Omega_{-}$ with the opposite sign  on the right hand side: 
		\begin{equation}\label{4.1111}
			\frac{1}{2}	\partial_{t} \int_{\Omega^{-}}(c^{2}|h^{-}|^{2}+|v^{-}|^{2} ) =-c^{2} \int_{\Gamma} h^{-} v^{-} e^{2}  .
		\end{equation}

		Adding \eqref{4.111} and \eqref{4.1111} implies 
		\begin{equation}\label{4.112}
			\frac{1}{2}	\partial_{t} \int_{\Omega}(c^{2} |h|^{2}+|v|^{2} ) =c^{2} \int_{\Gamma} [h v \cdot e_{2} ]= 2c^{2} \int_{\Gamma} h \bar{v}^{+}_{1} \partial_{1} f.
		\end{equation}
		Also multiplying the third equation in \eqref{2.2} by $f$, we have 
		\begin{equation}\label{4.3}
			\frac{1}{2}	\partial_{t}  \int_{\Gamma}|f|^{2} = \int_{\Gamma} v_{2} f.
		\end{equation}
		Adding \eqref{4.112} and \eqref{4.3} and using the Holder inequality  yields
		\begin{equation}\label{4.4}
			\begin{aligned}
				&\frac{1}{2}	\partial_{t} \int_{\Omega}(c^{2} |h|^{2}+|v|^{2}) +	\frac{1}{2}	\partial_{t}  \int_{\Gamma}|f|^{2}\\
				&= 2c^{2} \int_{\Gamma} h \bar{v}^{+}_{1} \partial_{1} f+\int_{\Gamma} v_{2} f  \\
				&\leq 2c^{2}\bar{v}^{+}_{1}  \|h\|_{L^{2}(\Gamma)} \|\partial_{1} f\|_{L^{2}(\Gamma)}+ \| v_{2} \|_{L^{2}(\Gamma)} \|f\|_{L^{2}(\Gamma)}:=J.
			\end{aligned}
		\end{equation}
		To avoid the loss of derivatives, we suppose that the solutions are band-limited at radius $R>0$, i.e., that 
		$$	\underset{x_{2}\in \mathbb{R}}{\cup}supp(|\hat{f}(\cdot)|+ |\hat{h}(\cdot,x_{2})|+|\hat{v}(\cdot,x_{2})|) \subset B(0,R),$$
		also we introduce an   anisotropic trace estimate in Lemma B.1 (\cite{WangXin}):
		\begin{equation}\label{4.114}
			\| \phi\|^{2}_{L^{2}(\Gamma)} \leq C(\|\bar{v}\cdot \nabla \phi\|_{L^{2}(\Omega)}\| \phi\|_{L^{2}(\Omega)}+ \| \phi\|^{2}_{L^{2}(\Omega)}),
		\end{equation}
		where $|\bar{v}|= \bar{v}^{+}_{1}\geq \varepsilon_{0}c>0$.
		Now we estimate $J$ as follows:
		\begin{equation}\label{4.115}
			\begin{aligned}
				J&\lesssim (\|\bar{v}^{+}_{1}\partial_{1} h\|_{L^{2}(\Omega)}\| h\|_{L^{2}(\Omega)}+ \| h\|^{2}_{L^{2}(\Omega)}) ^{\frac{1}{2}}  \|\eta \hat{f}\|_{L^{2}(\Gamma)}\\
				&+  (\|\bar{v}^{+}_{1}\partial_{1} v_{2}\|_{L^{2}(\Omega)}\| v_{2}\|_{L^{2}(\Omega)}+ \| v_{2}\|^{2}_{L^{2}(\Omega)})^{\frac{1}{2}} \|f\|_{L^{2}(\Gamma)}\\
				&\lesssim ((\bar{v}^{+}_{1}R+1)R \| h\|^{2}_{L^{2}(\Omega)}) ^{\frac{1}{2}}  \| f\|_{L^{2}(\Gamma)}+  ((\bar{v}^{+}_{1}R+1) \| v_{2}\|^{2}_{L^{2}(\Omega)})^{\frac{1}{2}} \|f\|_{L^{2}(\Gamma)}.
			\end{aligned}
		\end{equation}
		Finally plugging \eqref{4.115} into \eqref{4.4} and taking use of Gronwall's inequality, for arbitrary $R$, we have 
		\begin{equation}\label{4.116}
			\| f\|^{2}_{L^{2}(\Gamma)} + \| h\|^{2}_{L^{2}(\Omega)}+ \| v\|^{2}_{L^{2}(\Omega)}\leq C(\| f_{0}\|^{2}_{L^{2}(\Gamma)} + \| h_{0}\|^{2}_{L^{2}(\Omega)}+ \| v_{0}\|^{2}_{L^{2}(\Omega)}).
		\end{equation}
		From this, we infer that if $(f,h,v)|_{t=0}=0$, then it follows that $(f,h,v)\equiv0$. 
		
	\end{proof}
	
	\subsection{Discontinuous dependence on the initial data}
	\quad  In according  with Lemma 3.2 and \eqref{3.36}, if $0<\bar{v}^{+}_{1}<\sqrt{2}c$, we deduce that $X_{1}^{2}$ is positive, it follows that $\tau =    X_{1} \eta$  are real. Also we can infer that the equation \eqref{2.4} reduces to the following form
	\begin{equation}\label{4.5}
		\partial_{t}^{2}f + 	\lambda \partial^{2}_{11} f=0,
	\end{equation}
	where $\lambda$ must be positive in the case of $0<\bar{v}^{+}_{1}<\sqrt{2}c$. In fact, plugging $f= e^{\tau t} g$ into  equation \eqref{4.5}, we get $\tau^{2} g + 	\lambda \partial^{2}_{11} g=0$. therefore  the corresponding Fourier transform  form is
	\begin{equation}\label{4.6}
		(\tau^{2} -	\lambda \eta^{2}) \hat{g} =0,
	\end{equation}
	which yields   $\lambda = \frac{\tau^{2}}{\eta^{2}}$. From Lemma 3.2, we know that  $X_{1}^{2}=\frac{\tau^{2}}{\eta^2}>0$ in the case of $0<\bar{v}^{+}_{1}<\sqrt{2}c$. Thus we have $\lambda=X_{1}^{2}>0$. Therefore \eqref{4.5} is  a elliptic equation. Clearly   the solutions of \eqref{4.5} are  linear combination of the real and imaginary parts  of  function 
	\begin{equation}\label{4.7}
		f=e^{\sqrt{\lambda} k t} e^{i k x},
	\end{equation}
	where $k$ is a positive wave number.

	We are now in a position to prove  ill-posedness for this linear problem in the following lemma:
	\begin{lemm}
		In the case of $\epsilon_{0} \leq M:=\frac{\bar{v}^{+}_{1}}{c}<\sqrt{2}$,	the linear problem \eqref{2.2} with the correesponding jump boundary conditions \eqref{2.3} is ill-posed in the sense of Hadamard in $H^{k}(\Omega)$ for every $k$. More precisely, for any $k, j\in \mathbb{N}$ with $j\geq k$ and for any $T_{0}>0$ and $\alpha>0$ there exists a sequence $\{ (f_{n}, v_{n}, h_{n}) \}_{n=1}^{\infty}$ to \eqref{2.2}, satisfying jump boundary conditions \eqref{2.3}, so that
		\begin{equation}\label{4.8}
			\| (f_{n}(0), h_{n}(0), v_{n}(0))\|_{H^{j}} \lesssim \frac{1}{n},
		\end{equation}
		but
		\begin{equation}\label{4.9}
			\|(f_{n},h_{n}, v_{n})\|_{H^{k}}\geq \alpha,~for~all~t\geq T_{0}.
		\end{equation}
	\end{lemm}
	
	\begin{proof}
		For any $j\in \mathbb{N}$, we let $\chi_{n}(\eta)\in C_{c}^{\infty}(\mathbb{R})$ be a real-valued function so that  $supp(\chi_{n}) \subset B(0,n+1)\backslash B(0,n)$ and 
		\begin{equation} \label{4.10}
			\int_{\mathbb{R}} (1+ |\eta|^{2})^{j+1}	|\chi_{n}(\eta)|^{2} d \eta= \frac{1}{\bar{C_{j}}^{2} n^{2}}.
		\end{equation}
		We define 
		\begin{equation} \label{4.11}
			f_{n}(t,x_{1})= e^{\tau t} g_{n}(x_{1})= \frac{1}{4\pi^2}\int_{\mathbb{R}}e^{X_{1} \eta t} \chi_{n}(\eta)  e^{i \eta x_{1}}d \eta,
		\end{equation}
		which solves \eqref{4.5}. Here we make use of  $\tau=X_{1} \eta$ in according with Lemma 3.9, meanwhile we can see that  $\hat{g}_{n}=\chi_{n}(\eta)$. From this, we can see that the linearized front equation is qualitatively more unstable for large frequencies $\eta$. 	Since $\eta\rightarrow \infty$, the solutions \eqref{4.5} with  a higher frequency grow faster in time, which provides a mechanism for Kelvin-Helmholtz instability.   By the choice of $\chi_{n}$ and Plancherel theorem, we have the estimate
		\begin{equation} \label{4.12}
			\begin{aligned}
				&\|f_{n}(t=0,x_{1})\|_{H^{j}(\Gamma)}=\|g_{n}(x_{1})\|_{H^{j}(\Gamma)} \\
				&= (\int_{\mathbb{R}}  (1+ |\eta|^{2})^{j}	|\chi_{n}(\eta)|^{2} d \eta)^{1/2}\lesssim \frac{1}{n},
			\end{aligned}
		\end{equation}
		meanwhile  for $ n+1\geq \eta\geq n$ and $t\geq T_{0}$, we get
		\begin{equation} \label{4.13}
			\begin{aligned}
				\|f_{n}(t,x_{1})\|^{2}_{H^{k}(\Gamma)} &\geq e^{ 2X_{1} n T_{0}} \int_{\mathbb{R}}  (1+ |\eta|^{2})^{k}	|\chi_{n}(\eta)|^{2} d \eta\\
				&\geq  \frac{e^{2X_{1} n T_{0}}}{ (1+(n+1)^{2})^{j-k+1}}  \int_{\mathbb{R}}  (1+\eta^{2})^{j+1}    |\chi_{n}(\eta)|^{2} d\eta.
			\end{aligned}
		\end{equation}
		Let $n$ be sufficiently large so that 
		\begin{equation}\label{4.14}
			\frac{e^{2X_{1} n T_{0}}}{ (1+(n+1)^{2})^{j-k+1}} \geq \alpha ^{2}\bar{C_{j}}^{2} n^{2},
		\end{equation}
		thus we may estimate 
		\begin{equation}\label{4.15}
			\| f_{n}(t)\|_{H^{k}(\Gamma)}\geq \alpha.
		\end{equation}

		From the computation in section 2.2, we know that 
		\begin{equation} \label{4.16}
			\hat{m}_{n}(\eta,x_{2})=\left\{
			\begin{aligned}
				& \frac{4 i \bar{v}^{+}_{1} \eta \tau }{c^{2}(\mu^{+}+\mu^{-})}\hat{g}_{n}(\eta)e^{-\mu^{+}x_2}&x_2\ge 0,\\
				& \frac{4 i \bar{v}^{+}_{1} \eta \tau  }{c^{2}(\mu^{+}+\mu^{-})}\hat{g}_{n}(\eta)e^{\mu^{-}x_2}&x_2<0.
			\end{aligned}
			\right.
		\end{equation}
		Since  $\tau=X_{1}\eta>0$ and $\eta>0$, from lemma 3.1 we know that $\mu^{+}-\mu^{-}\neq 0$, then \eqref{4.16} can be rewritten as
		\begin{equation}\label{4.17}
			\hat{m}_{n}(\eta,x_{2})=\left\{
			\begin{aligned}
				& (\mu^{+}-\mu^{-})\hat{g}_{n}(\eta)e^{-\mu^{+}x_2}&x_2\ge 0,\\
				& (\mu^{+}-\mu^{-})\hat{g}_{n}(\eta)e^{\mu^{-}x_2}&x_2<0,
			\end{aligned}
			\right.
		\end{equation}
		here we note that $\mu^{\pm}$ only depend on $\eta$, since we get $\tau= X_{1}\eta$, therefore it implies that $\mu(\tau,\eta)= \mu(X_{1}\eta,\eta)$.

		By the Plancherel theorem and \eqref{4.16},  we have 
		\begin{equation}\label{4.18}
			\begin{aligned}
				&\| h_{n}(t,x_{1},x_{2})\|^{2}_{H^{k}(\Omega)}= \| e^{\tau t} m_{n}(x_{1},x_{2})\|^{2}_{H^{k}(\Omega)}\\
				&\geq    \int_{\mathbb{R}}  (1+\eta^{2})^{k}|\mu^{+}-\mu^{-}|^{2}|e^{\tau t} \hat{g}_{n}(\eta)|^{2} \int_{0}^{\infty}e^{-2\mu^{+}x_2} dx_{2} d \eta \\
				&+  \int_{\mathbb{R}}  (1+\eta^{2})^{k} |\mu^{+}-\mu^{-}|^{2} |e^{\tau t} \hat{g}_{n}(\eta)|^{2} \int_{-\infty}^{0}e^{2\mu^{-}x_2} dx_{2}  d \eta \\
				& \geq  \frac{1}{2}\int_{\mathbb{R}}  (1+\eta^{2})^{k} |\frac{\mu^{+}-\mu^{-}}{\mu^{+ }}|^{2} |\mu^{+}| e^{2 X_{1} \eta t} |\chi_{n}(\eta)|^{2} d\eta\\
				&+ \frac{1}{2}\int_{\mathbb{R}}  (1+\eta^{2})^{k}   |\frac{\mu^{+}-\mu^{-}}{\mu^{-}}|^{2} |\mu^{-}| e^{2 X_{1} \eta t} |\chi_{n}(\eta)|^{2} d \eta,
			\end{aligned}
		\end{equation}
		then we deduce that 
		\begin{equation}\label{4.19}
			|\frac{\mu^{+}-\mu^{-}}{\mu^{+ }}|^{2}= \frac{|2i\sqrt{\frac{r-a}{2}} |^{2}}{|\sqrt{\frac{r+a}{2}} + i\sqrt{\frac{r-a}{2}}|^{2}} =2\frac{r-a}{r}.
		\end{equation}
		Making using of $\epsilon_{0}\leq M:= \frac{\bar{v}^{+}_{1}}{c}< \sqrt{2}$ and $	X_{1}^{2}=\sqrt{c^{4}+ 4c^{2}(\bar{v}^{+}_{1})^{2}}-(\bar{v}^{+}_{1})^{2}-c^{2}$, we  get
		\begin{equation}\label{4.20}
			a= \frac{X^{2}_{1}- (\bar{v}^{+}_{1})^{2}+ c^{2}}{c^{2}} \eta^{2}= (\sqrt{1+4M^{2}}-2 M^{2})\eta^{2},
		\end{equation}
		where we estimate $a$ as follows:
		\begin{equation}\label{4.21}
			-\eta^{2}<a\leq  (\sqrt{1+4\epsilon_{0}^{2}}-2 \epsilon_{0}^{2})\eta^{2} ,~if~ \epsilon_{0}\leq M < \sqrt{2}.
		\end{equation}
		Also we compute
		\begin{equation}\label{4.22}
			\begin{aligned}
				&|\mu^{+}|= |\sqrt{\frac{r+a}{2}} + i\sqrt{\frac{r-a}{2}}| =\sqrt{ r},\\
				&|\mu^{-}|= |\sqrt{\frac{r+a}{2}} - i\sqrt{\frac{r-a}{2}}| = \sqrt{ r}.
			\end{aligned}
		\end{equation}
		In according with \eqref{3.36} and \eqref{3.41}, it implies that 
		\begin{equation}\label{4.23}
			\begin{aligned}
				r^{2}&=a^{2}+b^{2}\\
				&= \frac{X_{1}^{4}+ 2((\bar{v}^{+}_{1})^{2}+ c^{2})X_{1}^{2} +(\bar{v}^{+}_{1})^{4}- 2c^{2}(\bar{v}^{+}_{1})^{2} +c^{4} }{c^{4}} \eta^{4}=\eta^{4}.
			\end{aligned}
		\end{equation}
		Finally, combining with \eqref{4.19}, \eqref{4.20}, \eqref{4.21} and  \eqref{4.23} implies that 
		\begin{equation}\label{4.24}
			\tilde{C}:=2-2(\sqrt{1+4\epsilon_{0}^{2}}-2 \epsilon_{0}^{2}) \leq 	|\frac{\mu^{+}-\mu^{-}}{\mu^{+ }}|^{2} <4, ~if~ \epsilon_{0}\leq M < \sqrt{2},
		\end{equation}
		here we remark that  $\epsilon_{0}\leq M$ must be satisfied,  $\epsilon_{0}$ is a small but fixed number. Because if March number $M$ tend to zero, this lower bound tend to zero.
		
		Therefore, employing \eqref{4.24}, \eqref{4.17} and \eqref{4.22},    we estimate $\| h_{n}(0)\|_{H^{k}(\Omega)}$ as follows
		\begin{equation}\label{4.25}
			\begin{aligned}
				&\| h_{n}(t=0,x_{1},x_{2})\|^{2}_{H^{j}(\Omega)}= \|  m_{n}(x_{1},x_{2})\|^{2}_{H^{j}(\Omega)}\\
				&\leq   \sum_{s=0}^{j} \int_{\mathbb{R}}  (1+\eta^{2})^{j-s}|(\mu^{+}-\mu^{-})|^{2}| \hat{g}_{n}(\eta)|^{2} \int_{0}^{\infty}|\partial_{2}^{s} e^{-\mu^{+}x_2}|^{2} dx_{2} d \eta \\
				&+   \sum_{s=0}^{j}\int_{\mathbb{R}}  (1+\eta^{2})^{j-s} |(\mu^{+}-\mu^{-})|^{2} | \hat{g}_{n}(\eta)|^{2} \int_{-\infty}^{0}|\partial_{2}^{s}e^{\mu^{-}x_2}|^{2} dx_{2}  d \eta \\
				& \leq  \frac{1}{2} \sum_{s=0}^{j}\int_{\mathbb{R}}  (1+\eta^{2})^{j-s} |\frac{(\mu^{+}-\mu^{-})}{\mu^{+ }}|^{2} |\mu^{+}|^{2s+1}  |\chi_{n}(\eta)|^{2} d\eta\\
				&+ \frac{1}{2} \sum_{s=0}^{j} \int_{\mathbb{R}}  (1+\eta^{2})^{j-s}   |\frac{(\mu^{+}-\mu^{-})}{\mu^{-}}|^{2} |\mu^{-}|^{2s+1} |\chi_{n}(\eta)|^{2} d \eta\\
				& \leq 4(j+1) \int_{\mathbb{R}}  (1+\eta^{2})^{j+1}  |\chi_{n}(\eta)|^{2} d\eta  \lesssim \frac{1}{n}.
			\end{aligned}
		\end{equation}
		
		Meanwhile  for $\eta\geq n \geq 1$ and $t\geq T_{0}$, we may estimate \eqref{4.18} as follows
		\begin{equation}\label{4.26}
			\begin{aligned}
				\| h_{n}(t)\|^{2}_{H^{k}(\Omega)}& \geq  \tilde{C}\frac{e^{2 X_{1}n T_{0}}}{ 1+(n+1)^{j-k+1}}   \int_{\mathbb{R}}  (1+\eta^{2})^{j+1}    |\chi_{n}(\eta)|^{2} d\eta,
			\end{aligned}
		\end{equation}
		Let $n$ be sufficiently large so that 
		\begin{equation}\label{4.27}
			\tilde{C}\frac{e^{2 X_{1}n T_{0}}}{ 1+(n+1)^{j-k+1}} \geq \alpha^{2} n^{2} \bar{C}_{j}^{2}.
		\end{equation}
		Hence 
		we may estimate 
		\begin{equation}\label{4.28}
			\| h_{n}(t)\|_{H^{k}(\Omega)}\geq \alpha.
		\end{equation}
		
		Taking the horizontal Fourier transform of  the second equation in \eqref{2.2}, we arrive
		\begin{equation}\label{4.29}
			(\tau+i  \bar{v}_{1} \eta)  \hat{v}_{1}+ c^{2} i \eta  \hat{h}=0,
		\end{equation}
		and
		\begin{equation}\label{4.30}
			(\tau+i  \bar{v}_{1} \eta)  \hat{v}_{2}+ c^{2} \partial_{2} \hat{h}=0,
		\end{equation}
		we directly compute to find
		\begin{equation}\label{4.31}
			\hat{v}_{n,1}(t,\eta, x_{2})=\left\{
			\begin{aligned}
				& \frac{ (\mu^{+}-\mu^{-})c^{2} i \eta }{\tau+ i\bar{v}^{+}_{1} \eta }\hat{f}_{n}(t,\eta)e^{-\mu^{+}x_2}&x_2\ge 0,\\
				&\frac{ (\mu^{+}-\mu^{-})c^{2} i \eta }{\tau- i\bar{v}^{+}_{1} \eta }\hat{f}_{n}(t,\eta)e^{\mu^{-}x_2}&x_2<0,
			\end{aligned}
			\right.
		\end{equation}
		and
		\begin{equation}\label{4.32}
			\hat{v}_{n,2}(t,\eta, x_{2})=\left\{
			\begin{aligned}
				& \frac{(\mu^{+}-\mu^{-})c^{2} \mu^{+} }{(\tau+ i\bar{v}^{+}_{1} \eta )}\hat{f}_{n}(t,\eta)e^{-\mu^{+}x_2}&x_2\ge 0,\\
				& -\frac{(\mu^{+}-\mu^{-})c^{2} \tau\mu^{-} }{(\tau-i\bar{v}^{+}_{1} \eta )}\hat{f}_{n}(t,\eta)e^{\mu^{-}x_2}&x_2<0,
			\end{aligned}
			\right.
		\end{equation}
		then we may  estimate  $|\frac{i\eta}{\tau+ i\bar{v}^{+}_{1} \eta )} |$ and $|\frac{\mu^{+}}{\tau+ i\bar{v}^{+}_{1} \eta} |$ as follows:
		\begin{equation}\label{4.33}
			\begin{aligned}
				|\frac{i\eta}{\tau+ i\bar{v}^{+}_{1} \eta )} |^{2}&= \frac{1}{X^{2}_{1} + (\bar{v}^{+}_{1})^{2} }= \frac{1}{\sqrt{c^{4}+ 4c^{2}(\bar{v}^{+}_{1})^{2}}-c^{2}} \\
				&= \frac{1}{c^{2}  (\sqrt{1+4M^{2}}-1)}. 
			\end{aligned}
		\end{equation}
		Since  $\epsilon_{0}\leq M < \sqrt{2}$, we know that 
		\begin{equation}\label{4.34}
			\begin{aligned}
				\frac{1}{\sqrt{2}c} 	\leq 	|\frac{i\eta}{\tau+ i\bar{v}^{+}_{1} \eta )} |\leq  \frac{1}{c\sqrt{(\sqrt{1+4\epsilon_{0}^{2}}-1)}}.
			\end{aligned}
		\end{equation}

		The estimate \eqref{4.22} and \eqref{4.34} then imply that 
		\begin{equation}\label{4.35}
			\frac{1}{\sqrt{2}c} \leq 	|\frac{\mu^{+}}{\tau+ i\bar{v}^{+}_{1} \eta} | = \frac{1}{\sqrt{(X_{1}^{2}+ (\bar{v}^{+}_{1})^{2}) }}\leq  \frac{1}{c\sqrt{(\sqrt{1+4\epsilon_{0}^{2}}-1)}}.
		\end{equation}

		Therefore, employing \eqref{4.34} and \eqref{4.31},  we deduce 
		\begin{equation} \label{4.36}
			\begin{aligned}
				&\|v_{n,1}(t=0,x_{1},x_{2})\|^{2}_{H^{j}(\Omega)}= \|  w_{n}(x_{1},x_{2})\|^{2}_{H^{j}(\Omega)}\\
				&\leq   \sum_{s=0}^{j} \int_{\mathbb{R}}  (1+\eta^{2})^{j-s}|\frac{ (\mu^{+}-\mu^{-})c^{2} i \eta }{\tau+ i\bar{v}^{+}_{1} \eta }|^{2}\hat{g}_{n}(\eta)|^{2} \int_{0}^{\infty}|\partial_{2}^{s} e^{-\mu^{+}x_2}|^{2} dx_{2} d \eta \\
				&+  \sum_{s=0}^{j} \int_{\mathbb{R}}  (1+\eta^{2})^{j-s} |\frac{ (\mu^{+}-\mu^{-})c^{2} i \eta }{\tau- i\bar{v}^{+}_{1} \eta }|^{2}| \hat{g}_{n}(\eta)|^{2} \int_{-\infty}^{0}|\partial_{2}^{s}e^{\mu^{-}x_2}|^{2} dx_{2}  d \eta \\
				& \leq  \frac{1}{2c^{2}(\sqrt{1+4\epsilon_{0}^{2}}-1)}  \sum_{s=0}^{j}\int_{\mathbb{R}}  (1+\eta^{2})^{j-s} |\frac{(\mu^{+}-\mu^{-})}{\mu^{+ }}|^{2} |\mu^{+}|^{2s+1}  |\chi_{n}(\eta)|^{2} d\eta\\
				&+ \frac{1}{2c^{2}(\sqrt{1+4\epsilon_{0}^{2}}-1)} \sum_{s=0}^{j}\int_{\mathbb{R}}  (1+\eta^{2})^{j-s}   |\frac{(\mu^{+}-\mu^{-})}{\mu^{-}}|^{2} |\mu^{-}|^{2s+1} |\chi_{n}(\eta)|^{2} d \eta\\
				& \leq \frac{j+1}{c^{2}(\sqrt{1+4\epsilon_{0}^{2}}-1)} \int_{\mathbb{R}}  (1+\eta^{2})^{j+1}  |\chi_{n}(\eta)|^{2} d\eta  \lesssim  \frac{1}{ n^{2}}.
			\end{aligned}
		\end{equation}
		Similarly, we have 
		\begin{equation}\label{4.37}
			\|v_{n,2}(t=0,x_{1},x_{2})\|^{2}_{H^{j}(\Omega)} \lesssim  \frac{1}{ n^{2}},
		\end{equation}
		whereas for $\eta\geq n$ and $t\geq T_{0}$ we  deduce
		\begin{equation} \label{4.38}
			\begin{aligned}
				&\|v_{n,1}\|_{H^{k}(\Omega)}^{2}\geq  \int_{\mathbb{R}}(1+\eta^{2})^{k} \|\hat{v}_{1}\|_{L^{2}(I_{\pm})}^{2} d \eta\\
				&\geq  \int_{\mathbb{R}}  (1+\eta^{2})^{k} |\frac{ (\mu^{+}-\mu^{-})c^{2} i \eta }{\tau+ i\bar{v}^{+}_{1} \eta }|^{2}|\hat{f}|^{2}e^{-\mu^{+}x_2}\int_{0}^{\infty}e^{-2\mu^{+}x_2} dx_{2} d \eta\\
				&+\int_{\mathbb{R}}  (1+\eta^{2})^{k} |\frac{ (\mu^{+}-\mu^{-})c^{2} i \eta }{\tau- i\bar{v}^{+}_{1} \eta }|^{2} |\hat{f}|^{2}\int_{0}^{\infty}e^{2\mu^{-}x_2} dx_{2} d \eta\\
				&\geq  \frac{1}{2}\int_{\mathbb{R}}  (1+\eta^{2})^{k} c^{4}|\frac{ \mu^{+}-\mu^{-}}{\mu^{+}}|^{2} |\frac{i \eta }{\tau+ i\bar{v}^{+}_{1} \eta }|^{2} e^{2 X_{1} \eta t} |\chi_{n}(\eta)|^{2} |\mu^{+}|d \eta \\
				&+\frac{1}{2}\int_{\mathbb{R}}  (1+\eta^{2})^{k} c^{4}|\frac{ \mu^{+}-\mu^{-}}{\mu^{-}}|^{2} |\frac{i \eta }{\tau- i\bar{v}^{+}_{1} \eta }|^{2} e^{2 X_{1} \eta t} |\chi_{n}(\eta)|^{2} |\mu^{-}|d \eta\\
				&\geq c^{2}\tilde{C}\frac{e^{2 X_{1}n T_{0}}}{ 1+(n+1)^{j-k+1}}  n \int_{\mathbb{R}}  (1+\eta^{2})^{j+1}    |\chi_{n}(\eta)|^{2} d\eta.
			\end{aligned}
		\end{equation}
		
		Let $n$ be sufficiently large so that 
		\begin{equation}\label{4.39}
			c^{2} \tilde{C}\frac{e^{2 X_{1}n T_{0}}}{ 1+(n+1)^{j-k+1}} \geq \alpha^{2} n \bar{C}_{j}^{2}.
		\end{equation}
		Hence 
		we may estimate 
		\begin{equation}\label{4.40}
			\| v_{n,1}(t)\|_{H^{k}(\Omega)}\geq \alpha.
		\end{equation}
		Similarly we have 
		\begin{equation} \label{4.41}
			\|v_{n,2}\|_{H^{k}(\Omega)}^{2}\geq  \alpha.
		\end{equation}
		
		Collecting the estimates \eqref{4.12}, \eqref{4.25} and \eqref{4.36} gives 
		\begin{equation}\label{4.42}
			\|f_{n}(0)\|_{H^{j}(\Omega)}+\|h_{n}(0)\|_{H^{j}(\Gamma)}+ \|v_{n}(0)\|_{H^{j}(\Omega)}\lesssim \frac{1}{n},
		\end{equation}
		but the estimates \eqref{4.15}, \eqref{4.28} \eqref{4.40} and \eqref{4.41} yield
		\begin{equation}\label{4.43}
			\|f_{n}\|_{H^{k}(\Gamma)}+ \|h_{n}\|_{H^{k}(\Omega)}+ \| v_{n}\|_{H^{k}(\Omega)}\geq \alpha,~for~all~t\geq T_{0}.
		\end{equation}
	\end{proof}

	\section{Ill-posedness  for the nonlinear problem}
	\quad \quad  Now  we will prove nonlinear ill-posedness for the nonlinear problem \eqref{1.10}.  To begin with, we rewrite the nonlinear system \eqref{1.10} in a perturbation formulation around the rectilinear solution.  Let
	\begin{equation}\label{5.1}
		\begin{aligned}
			f=0+ \tilde{f},~v= \bar{v}+ \tilde{v},~h= \bar{h}+ \tilde{h},\Psi= Id+ \tilde{\Psi}, \\
			\varrho=\bar{\varrho}+ \tilde{\varrho},~\psi= 0+ \tilde{\psi},~ n=e_{2}+ \tilde{n},~     A=I-B,
		\end{aligned}
	\end{equation}
	where
	\begin{equation}\label{5.2}
		B^{T}= \sum_{n=1}^{\infty} (-1)^{n-1} (D\tilde{\Psi} )^{n}.
	\end{equation}
	We can rewrite the term $\breve{v}$ as follows
	\begin{equation}\label{5.3}
		\begin{aligned}
			\breve{v}&= (I-B)(\bar{v}+ \tilde{v})- (0, \frac{\partial_{t}\tilde{\psi}}{1+ \partial_{2} \tilde{\psi}})\\
			&= \bar{v}+ \tilde{v}- B(\bar{v}+ \tilde{v})- (0, \frac{\partial_{t}\tilde{\psi}}{1+ \partial_{2} \tilde{\psi}}):= \bar{v}+M,
		\end{aligned}
	\end{equation}
	where the $M$ is defined as follows
	\begin{equation}\label{5.4}
		\begin{aligned}
			M= \tilde{v}- B(\bar{v}+ \tilde{v})- (0, \frac{\partial_{t}\tilde{\psi}}{1+ \partial_{2} \tilde{\psi}}).
		\end{aligned}
	\end{equation}
	To linearized  the term $ c^{2}(h)= c^{2}(\bar{h}+ \tilde{h})$,  we employ Taylor formula to get
	\begin{equation}\label{5.5}
		c^{2}(\bar{h}+ \tilde{h})= c^{2}(\bar{h}) + \mathcal{R},
	\end{equation}
	where the reminder term is defined by
	\begin{equation}\label{5.6}
		\begin{aligned}
			\mathcal{R}=  (c^{2})^{\prime}(\bar{h}+ (1-\alpha)\tilde{h}) \tilde{h}, ~0<\alpha<1.
		\end{aligned}
	\end{equation}
	For the term $v\cdot n $,  we can rewrite it as
	\begin{equation}\label{5.7}
		v\cdot n= (\bar{v}+ \tilde{v})\cdot (e_{2}+ \tilde{n})
		=\tilde{v}_{2}-\bar{v}_{1} \partial_{1} \tilde{f}+\tilde{v} \cdot \tilde{n}.
	\end{equation}
	Then the nonlinear system \eqref{1.10}  can be rewritten for $\tilde{h}, \tilde{v}, \tilde{f}$ as
\begin{equation}\label{5.8}
	\begin{cases}
		\partial_{t}\tilde{h}+(\bar{v} \cdot \nabla) \tilde{h}+ \nabla \cdot \tilde{v}=-(M \cdot \nabla) \tilde{h}+B^{T} \nabla \cdot \tilde{v}, & \text { in }[0, T] \times \Omega, \\
		\partial_{t} \tilde{v}+(\bar{v} \cdot \nabla) \tilde{v}
		+ c^{2}(\varrho_{0})\nabla \tilde{h}=-(M \cdot \nabla) \tilde{v}\\
		+c^{2}(\varrho_{0})B^{T}\nabla \tilde{h}- \mathcal{R}(\nabla \tilde{h}- B^{T}\nabla \tilde{h}), & \text { in }[0, T] \times \Omega, \\
		\partial_{t} \tilde{f}+\bar{v}_{1} \partial_{1} \tilde{f}-\tilde{v}_{2} =\tilde{v} \cdot \tilde{n},   & \text { on }[0, T] \times \Gamma.
	\end{cases}
\end{equation}
The jump conditions take new form in terms of $\tilde{h}, \tilde{v}, \tilde{f}$
\begin{equation}\label{5.9}
	\left\{\begin{array}{l}
		\left(\tilde{v}^{+}-\tilde{v}^{-}\right) \cdot e_{2}+(\bar{v}^{+}-\bar{v}^{-})\cdot \tilde{n} =-(\tilde{v}^{+}-\tilde{v}^{-})\cdot \tilde{n}, \\
		\bar{h}^{+}+  \tilde{h}^{+}=\bar{h}^{-}+  \tilde{h}^{-}.
	\end{array}\right.
\end{equation}

\underline{\textbf{Proof of Theorem \ref{theorem:main}}}\quad 
Now we are ready to prove the main theorem 1.3.   We prove it by the method of  contradiction. Suppose that the system \eqref{1.10}  is locally well-posedness  for some $k \geq 3$. Let $\delta, t_{0}, C>0$ be the constants provided by Definition 1.2. For $\varepsilon>0$, let $(f^{\varepsilon},h^{\varepsilon},v^{\varepsilon})(t)$ with initial data $(f^{\varepsilon},h^{\varepsilon},v^{\varepsilon})|_{t=0}=(f^{\varepsilon}_{0},h^{\varepsilon}_{0},v^{\varepsilon}_{0})$ is an sequence  solution of the system \eqref{1.10}. We choose $(f^{1},h^{1},v^{1}) $ to be  $(f^{\varepsilon},h^{\varepsilon},v^{\varepsilon})$.  We also replace $(f^{2}_{0},h^{2}_{0},v^{2}_{0}) $ by a steady-state solution $U\equiv (\bar{f},\bar{h},\bar{v})$. Obviously, $U$ is always the solution of the system \eqref{1.10}. For simplicity, we always take this steady-state $U$ as the solution of the system \eqref{1.10}, i.e.,  $(f^{2},h^{2},v^{2})(t) =U$ for $t\geq 0$.

Fix $n \in \mathbb{N}$ so that $n>C$. Applying Lemma 4.2 with this $n, T_{0}=t_{0} / 2, k \geq 3$, and $\alpha=2$, we can find $f^{L}, h^{L}, v^{L}$ solving \eqref{2.2} so that
\begin{equation}\label{5.12}
	\|(f^{L}_{0},h^{L}_{0}, v^{L}_{0})\|_{H^{k}}\lesssim \frac{1}{n},
\end{equation}
but
\begin{equation}\label{5.13}
	\|(f^{L}(t),h^{L}(t), v^{L}(t))\|_{H^{3}} \geq 2 \quad \text { for } t \geq t_{0} / 2.
\end{equation}

We define $\tilde{f}_{0}^{\varepsilon}=f_{0}^{\varepsilon}-\bar{f} :=\varepsilon f^{L}_{0}$, $\tilde{h}_{0}^{\varepsilon}=h_{0}^{\varepsilon}-\bar{h}:=\varepsilon h^{L}_{0}$  and $\tilde{v}_{0}^{\varepsilon}=v_{0}^{\varepsilon}-\bar{v}:=\varepsilon v^{L}_{0}$. Then for $\varepsilon<\delta n$ we have $\|(\tilde{f}_{0}^{\varepsilon}, \tilde{h}_{0}^{\varepsilon},  \tilde{v}_{0}^{\varepsilon})\|_{H^{k}}<\delta$, so according to Definition 1.2, there exist $\left(\tilde{f}^{\varepsilon}:=f^{\varepsilon}- \bar{f} , \tilde{h}^{\varepsilon}:=h^{\varepsilon}-\bar{h}, \tilde{v}^{\varepsilon}:= v^{\varepsilon}-\bar{v} \right) \in L^{\infty}\left(\left[0, t_{0}\right] ; H^{3}(\Omega)\right)$ that solve \eqref{5.8}-\eqref{5.9} with $(\tilde{f}_{0}^{\varepsilon}, \tilde{h}_{0}^{\varepsilon}, \tilde{v}_{0}^{\varepsilon})$ as initial data and that satisfy the inequality
\begin{equation}\label{5.14}
	\begin{aligned}
		\sup _{0 \leq t \leq t_{0}}\left\|\left(\tilde{f}^{\varepsilon}, \tilde{h}^{\varepsilon}, \tilde{v}^{\varepsilon} \right)(t)\right\|_{H^{3}} & \leq C\left(\left\|\left(f_{0}^{\varepsilon}, h_{0}^{\varepsilon}, v_{0}^{\varepsilon}\right)\right\|_{H^{k}}\right)  \\
		& \leq C \varepsilon \frac{1}{n}<\varepsilon.
	\end{aligned}
\end{equation}

Now define the rescaled functions $\bar{f}^{\varepsilon}=\tilde{f}^{\varepsilon} / \varepsilon, \bar{h}^{\varepsilon}=\tilde{h}^{\varepsilon} / \varepsilon , \bar{v}^{\varepsilon}=\tilde{v}^{\varepsilon} / \varepsilon$; rescaling \eqref{5.14} then shows that
\begin{equation}\label{5.15}
	\sup _{0 \leq t \leq t_{0}}\left\|(\bar{f}^{\varepsilon}, \bar{h}^{\varepsilon}, \bar{v}^{\varepsilon})(t)\right\|_{H^{3}}<1.
\end{equation}

By construction, we know that  $(\bar{f}^{\varepsilon}_{0}, \bar{h}^{\varepsilon}_{0}, \bar{v}^{\varepsilon}_{0})=(f^{L}_{0}, h^{L}_{0}, v^{L}_{0})$. We are going  to show that the rescaled functions  $(\bar{f}^{\varepsilon}, \bar{h}^{\varepsilon}, \bar{v}^{\varepsilon})$ converge as $\varepsilon \rightarrow 0$ to the solutions $(f^{L}, h^{L}, v^{L})$ of the linearized equations \eqref{3.2}.

Now we are going to  reformulate \eqref{5.8}-\eqref{5.9} in terms of rescaled functions $(\bar{f}^{\varepsilon}, \bar{h}^{\varepsilon}, \bar{v}^{\varepsilon})$ and   show  some convergence results. The third equation in \eqref{5.5} can rewritten in terms of rescaled function $(\bar{f}^{\varepsilon}, \bar{h}^{\varepsilon}, \bar{v}^{\varepsilon})$  as follows:
\begin{equation}\label{5.16}
	\partial_{t} \bar{f} ^{\varepsilon}+\bar{v}_{1} \partial_{1} \bar{f}^{\varepsilon}-\bar{v}^{\varepsilon}_{2} =\varepsilon \bar{v}^{\varepsilon} \cdot n^{\varepsilon}. 
\end{equation}
where  $n^{\varepsilon}=\frac{(-\varepsilon \partial_{1} \bar{f}^{\varepsilon}, 0)}{\varepsilon}=(- \partial_{1} \bar{f}^{\varepsilon}, 0)$ is well defined and uniformly bounded in $L^{\infty}\left(\left[0, t_{0}\right] ; H^{2}(\Gamma)\right)$ since
\begin{equation}\label{5.17}
	\|n^{\varepsilon}\|_{H^{2}(\Gamma)}\leq \| \bar{f}^{\varepsilon}\|_{H^{3}(\Gamma)}<1.
\end{equation}
Hence from \eqref{5.15} and\eqref{5.17}, we obtain
\begin{equation}\label{5.18}
	\lim _{\varepsilon \rightarrow 0} \sup _{0 \leq t \leq t_{0}}\left\|\partial_{t} \bar{f}^{\varepsilon}+\bar{v}_{1} \partial_{1} \bar{f}^{\varepsilon}-\bar{v}^{\varepsilon}_{2}\right\|_{H^{2}}=0 
\end{equation}
and 
\begin{equation}\label{5.188}
	\sup _{0 \leq t \leq t_{0}}\left\|\partial_{t} \bar{f}^{\varepsilon}(t)\right\|_{H^{2}}\leq \bar{v}_{1} \sup _{0 \leq t \leq t_{0}}\left\|\partial_{1} \bar{f}^{\varepsilon}(t)\right\|_{H^{2}}+\sup _{0 \leq t \leq t_{0}}\left\|\bar{v}^{\varepsilon}_{2}\right\|_{H^{2}}\leq C
\end{equation}

Expanding the first equation in \eqref{5.8}  implies that
\begin{equation}\label{5.19}
	\partial_{t}\bar{h}^{\varepsilon}+(\bar{v} \cdot \nabla) \bar{h}^{\varepsilon}+ \nabla \cdot \bar{v}^{\varepsilon}=-\varepsilon(M^{\varepsilon} \cdot \nabla) \bar{h}^{\varepsilon}+\varepsilon(B^{\varepsilon})^{T} \nabla \cdot \bar{v}^{\varepsilon}, 
\end{equation}
where we define $M^{\varepsilon}$ as follows
\begin{equation}\label{5.20}
	M^{\varepsilon}=\bar{v}^{\varepsilon}-B^{\varepsilon}(\bar{v}+ \varepsilon \bar{v}^{\varepsilon})-(0,\frac{\partial_{t} \psi^{\varepsilon}}{1+ \varepsilon \partial_{2} \psi^{\varepsilon}}), \psi^{\varepsilon}= \theta \bar{f}^{\varepsilon}.
\end{equation}
In order to estimate the bound of $M^{\varepsilon}$, we firstly estimate the bound of $B^{\varepsilon}$. We  assume that $\varepsilon$ is sufficiently small so that
$\varepsilon<1 /\left(2 C_{1}\right)$, where $K_{1}>0$ is the best constant in the inequality $\|U V\|_{H^{2}} \leq$ $C_{1}\|U\|_{H^{2}}\|V\|_{H^{2}}$ for $3 \times 3$ matrix-valued functions $U, V$. This assumption guarantees that $B^{\varepsilon}:=(I-(I+\varepsilon \nabla  \Psi^{\varepsilon})^{-1})/\varepsilon $ is well defined and uniformly bounded in $L^{\infty}\left(\left[0, t_{0}\right] ; H^{2}(\Omega)\right)$ since
\begin{equation}\label{5.21}
	\begin{aligned}
		\left\|{B}^{\varepsilon}\right\|_{H^{2}} & =\left\|\sum_{n=1}^{\infty}(-\varepsilon)^{n-1}(\nabla  \Psi^{\varepsilon})^{n}\right\|_{H^{2}} \leq \sum_{n=1}^{\infty} \varepsilon^{n-1}\left\|(\nabla  \Psi^{\varepsilon})^{n}\right\|_{H^{2}}  \\
		& \leq \sum_{n=1}^{\infty}\left(\varepsilon K_{1}\right)^{n-1}\left\|\nabla  \Psi^{\varepsilon}\right\|_{H^{2}}^{n} \leq \sum_{n=1}^{\infty} \frac{1}{2^{n-1}}\left\|\psi^{\varepsilon}\right\|_{H^{3}}^{n}\\
		&\leq \sum_{n=1}^{\infty} \frac{1}{2^{n-1}}\left\|\bar{f}^{\varepsilon}\right\|_{H^{3}}^{n}
		<\sum_{n=1}^{\infty} \frac{1}{2^{n-1}}=2,
	\end{aligned}
\end{equation}
whereas we shows that 
\begin{equation}\label{5.22}
	\begin{aligned}
		\left\|{M}^{\varepsilon}\right\|_{H^{2}}& \leq \| \bar{v}^{\varepsilon}\|_{H^{2}}+  \bar{v} \|  B^{\varepsilon}\|_{H^{2}} \|+\varepsilon \|  B^{\varepsilon}\|_{H^{2}}\|\bar{v}^{\varepsilon}\|_{H^{2}}+\|\partial_{t}  \psi^{\varepsilon}\|_{H^{2}}\\
		&\leq \| \bar{v}^{\varepsilon}\|_{H^{2}}+   \bar{v} \|  B^{\varepsilon}\|_{H^{2}} \|+\varepsilon \|  B^{\varepsilon}\|_{H^{2}}\|\bar{v}^{\varepsilon}\|_{H^{2}}+\|\partial_{t}  \bar{f}^{\varepsilon}\|_{H^{2}}\\
		&\leq C.
	\end{aligned}
\end{equation}

Therefore employing \eqref{5.15},\eqref{5.21} and \eqref{5.22} we get
\begin{equation}\label{5.23}
	\lim _{\varepsilon \rightarrow 0}	\sup _{0 \leq t \leq t_{0}}\left\|\partial_{t} \bar{h}^{\varepsilon}+(\bar{v} \cdot \nabla) \bar{h}^{\varepsilon}+ \nabla \cdot \bar{v}^{\varepsilon} \right\|_{H^{2}}=0, 
\end{equation}
and 
\begin{equation}\label{5.24}
	\sup _{0 \leq t \leq t_{0}}\left\|\partial_{t} \bar{h}^{\varepsilon}(t)\right\|_{H^{2}}<C.
\end{equation}

Expanding the second equation in \eqref{5.8}, we find that
\begin{equation}\label{5.25}
	\begin{aligned}
		&\partial_{t} \bar{v}^{\varepsilon}+(\bar{v} \cdot \nabla) \bar{v}^{\varepsilon}+ c^{2}\nabla \bar{h}^{\varepsilon}=-\varepsilon(M^{\varepsilon} \cdot \nabla) \bar{v}^{\varepsilon}\\
		&+\varepsilon c^{2}(B^{\varepsilon})^{T}\nabla \bar{h}^{\varepsilon}+ \varepsilon\mathcal{R}^{\varepsilon}(\nabla  \bar{h}^{\varepsilon}- \varepsilon(B^{\varepsilon})^{T}\nabla  \bar{h}^{\varepsilon}). 
	\end{aligned}
\end{equation}
where  we define the normalized remainder function by
\begin{equation}\label{5.26}
	\mathcal{R}^{\varepsilon}(x, t)  =\frac{ (c^{2})^{\prime}(\bar{h}+ (1-\alpha) \varepsilon \bar{h}^{\varepsilon}) \varepsilon \bar{h}^{\varepsilon}}{\varepsilon}= (c^{2})^{\prime}(\bar{h}+ (1-\alpha) \varepsilon \bar{h}^{\varepsilon})\bar{h}^{\varepsilon}.
\end{equation}
It is easier to show  that $\bar{h}+ (1-\alpha) \varepsilon \bar{h}^{\varepsilon}$ is bounded above  by a positive constant. Taking use of  \eqref{5.15}  which imply
\begin{equation}\label{5.28}
	\sup _{0 \leq t \leq t_{0}}\left\|	\mathcal{R}^{\varepsilon}(x, t)\right\|_{H^{3}} \leq C.
\end{equation}

Therefore from \eqref{5.28} and \eqref{5.15}, we deduce that 
\begin{equation}\label{5.29}
	\lim _{\varepsilon \rightarrow 0} \sup _{0 \leq t \leq t_{0}}\left\|\partial_{t} \bar{v}^{\varepsilon}+(\bar{v} \cdot \nabla) \bar{v}^{\varepsilon}+ c^{2}(\varrho_{0})\nabla \bar{h}^{\varepsilon}\right\|_{H^{2}}=0 
\end{equation}
and
\begin{equation}\label{5.30}
	\sup _{0 \leq t \leq t_{0}}\left\|\partial_{t} \bar{v}^{\varepsilon}(t)\right\|_{H^{2}}<C.
\end{equation}

Next, we deal with some convergence results for the jump conditions. For the first equation in \eqref{5.9} we rewrite the normal vector $n$  as follows
\begin{align*}\label{5.36}
	n=e_{2}+ \tilde{n}^{\varepsilon} : =e_{2}+\varepsilon n^{\varepsilon},~n^{\varepsilon}=(-\partial_{1} \bar{f}^{\varepsilon},0),
\end{align*}
so we may rewrite the second equation in \eqref{5.9} as
\begin{equation}\label{5.37}
	\left(\bar{v}^{+}+\varepsilon \bar{v}^{+,\varepsilon}-\bar{v}^{-}-\varepsilon \bar{v}^{-,\varepsilon}\right) \cdot\left(e_{2}+\varepsilon n^{\varepsilon}\right)=0 
\end{equation}
Since  $\sup_{0 \leq t \leq t_{0}}\left\|n^{\varepsilon}(t)\right\|_{L^{\infty}}\leq \| n^{\varepsilon}\|_{H^{3}(\Gamma)}<1 $ is bounded uniformly, we find that
\begin{equation}\label{5.38}
	\sup _{0 \leq t \leq t_{0}}\left\|e_{2} \cdot\left(\bar{v}^{+,\varepsilon}(t)-\bar{v}^{-,\varepsilon}(t)+ (\bar{v}^{+}-\bar{v}^{-}) \cdot  n^{\varepsilon} \right)\right\|_{L^{\infty}} \rightarrow 0 \quad \text { as } \varepsilon \rightarrow 0 
\end{equation}
Therefore we have 
\begin{equation}\label{5.39}
	[\bar{v}^{\varepsilon} \cdot e_{2}]= 2\bar{v}^{+}_{1} \partial_{1} f^{\varepsilon}~on~ \Gamma.
\end{equation}

We expand the second equation in \eqref{5.9} as follows
\begin{equation}\label{5.31}
	\bar{h}^{+}+ \varepsilon \bar{h}^{+,\varepsilon}=\bar{h}^{-}+ \varepsilon \bar{h}^{-,\varepsilon} ~on~ \Gamma.
\end{equation}
Since $\bar{h}^{+}=\bar{h}^{-}$, we may eliminate these two terms from equation \eqref{5.31} and divide both sides by $\varepsilon$ to get
\begin{equation}\label{5.35}
	\bar{h}_{+}^{\varepsilon}=\bar{h}_{-}^{\varepsilon} ~on~ \Gamma.
\end{equation}

According to the bound \eqref{5.15} and sequential weak-* compactness, we have that up to the extraction of a subsequence (which we still denote using only $\varepsilon$ )
\begin{equation}\label{5.40}
	(\bar{f}^{\varepsilon}, \bar{h}^{\varepsilon}, \bar{v}^{\varepsilon}) \stackrel{*}\rightharpoonup (f^{\star}, h^{\star}, v^{\star}) \quad \text { weakly }-* \text { in } L^{\infty}\left(\left[0, t_{0}\right] ; H^{3}(\Omega)\right). 
\end{equation}
By lower semicontinuity we know that
\begin{equation}\label{5.41}
	\sup _{0 \leq t \leq t_{0}}\left\|(f^{\star},h^{\star},v^{\star})(t)\right\|_{H^{3}} \leq 1. 
\end{equation}
In according with \eqref{5.188}, \eqref{5.24}, and \eqref{5.30}, we get
\begin{equation}\label{5.42}
	\limsup _{\varepsilon \rightarrow 0} \sup _{0 \leq t \leq t_{0}}\left\|(\partial_{t} \bar{f}^{\varepsilon}, \partial_{t} \bar{h}^{\varepsilon}, \partial_{t} \bar{v}^{\varepsilon})(t)\right\|_{H^{2}}<\infty.
\end{equation}

By Lions-Abin lemma in \cite{Simon}, we then have that the sequence $\left\{\left(f^{\varepsilon}, h^{\varepsilon}, v^{\varepsilon}\right)\right\}$ is strongly precompact in the space $L^{\infty}\left(\left[0, t_{0}\right] ; H^{8/3}(\Omega)\right)$, so
\begin{equation}\label{5.43}
	(\bar{f}^{\varepsilon},\bar{h}^{\varepsilon}, \bar{v}^{\varepsilon}) \rightarrow (f^{\star}, h^{\star}, v^{\star})\quad \text { strongly in } L^{\infty}\left(\left[0, t_{0}\right] ; H^{8 / 3}(\Omega)\right). 
\end{equation}
This strong convergence, together with  \eqref{5.18}, \eqref{5.24},\eqref{5.30}, implies that
\begin{equation}\label{5.44}
	(\partial_{t} \bar{f}^{\varepsilon}, \partial_{t} \bar{h}^{\varepsilon}, \partial_{t} \bar{v}^{\varepsilon}) \rightarrow\left(\partial_{t} f^{\star}, \partial_{t} v^{\star}, \partial_{t} h^{\star}\right) \text { strongly in } L^{\infty}\left(\left[0, t_{0}\right] ; H^{5 /3}(\Omega)\right),
\end{equation}
The index $\frac{8}{3}$ and $\frac{5}{3}$ are sufficient large to give $L^{\infty}([0,t_{0}];L^{\infty})$ convergence of $\left\{\left(f^{\varepsilon}, h^{\varepsilon}, v^{\varepsilon}\right)\right\}$, thus we have 
\begin{equation}\label{5.45}
	\left\{
	\begin{aligned}
		&\partial_{t}h^{\star}+(\bar{v} \cdot \nabla) h^{\star}+ \nabla \cdot v^{\star}=0,& \text { in } \Omega, \\
		&\partial_{t} v^{\star}+(\bar{v} \cdot \nabla) v^{\star}+ c^{2}\nabla h^{\star}=0,& \text { in } \Omega, \\ 
		&\partial_{t} f^{\star}+\bar{v}_{1} \partial_{1} f^{\star}-v^{\star}_{2}=0& \text { on } \Gamma,
	\end{aligned}
	\right.
\end{equation}
and
\begin{equation}\label{5.46}
	\begin{aligned}
		h_{+}^{\star}=h^{\star}_{-} & \text { on }~\Gamma, \\
		\left(v_{+}^{\star}-v_{-}^{\star}\right) \cdot e_{2}= 2\bar{v}^{+}_{1} \partial_{1} f^{\star} & \text { on }~\Gamma.
	\end{aligned}
\end{equation}
We also pass to the limit in the initial conditions $(\bar{f}^{\varepsilon}_{0}, \bar{h}^{\varepsilon}_{0}, \bar{v}^{\varepsilon}_{0})=(f^{L}_{0}, h^{L}_{0}, v^{L}_{0})$ to obtain
\begin{equation*}
	(f^{\star}_{0}, v^{\star}_{0}, h^{\star}_{0}) =(f^{L}_{0}, h^{L}_{0}, v^{L}_{0}).
\end{equation*}
Now we can see that $(f^{\star}, v^{\star}, h^{\star})(t)$ are solutions to \eqref{2.2} and boundary conditions \eqref{2.3} with same initial data. In according with the uniqueness result in lemma 4.1, we
have
\begin{equation}\label{5.47}
	(f^{\star}, v^{\star}, h^{\star})(t) =(f^{L}, v^{L}, h^{L})(t). 
\end{equation}
Therefore  we combine  inqualities \eqref{5.41} with \eqref{5.15} to get
\begin{equation}\label{5.48}
	2=\alpha<\sup _{0 \leq t \leq t_{0}}\left\|(f^{\star}, v^{\star}, h^{\star})(t)\right\|_{H^{3}} \leq 1. 
\end{equation}
which is a contradiction. Therefore, the proof of Theorem 1.3 is completed.

	\section{Acknowledgements}
	\quad The  authors would like to thank Professor Zhouping Xin for pointing this direction. The  authors also would like to thank Professor Yan Guo for his encouragement. Xie's work is supported by NSF-China under grant number 11901207 and the National Key Program  of China (2021YFA1002900). Zhao's  work is supported by the project funded by China Postdoctoral Science Foundation (2023M740334).

	\bibliographystyle{amsplain}

\end{document}